\documentclass[preprint,12pt,sort&compress,numbers]{elsarticle}
\usepackage{stmaryrd}
\usepackage{amsfonts}
\usepackage{amsmath}
\usepackage{hyperref}
\usepackage{amssymb}
\usepackage{amsthm}
\usepackage{mathrsfs}
\usepackage{appendix}

\input amssym.def
\input amssym.tex

\date{}

\theoremstyle{plain}
\numberwithin{equation}{section}
\newtheorem{Theorem}{Theorem}[section]
\newtheorem{Proposition}{Proposition}[section]
\newtheorem{Lemma}{Lemma}[section]
\newtheorem{Corollary}{Corollary}[section]
\newtheorem{Definition}{Definition}[section]
\theoremstyle{definition}
\newtheorem{Remark}{Remark}[section]

\newcommand{\definition}{{\lower .5ex
\hbox{$\>\>\stackrel{\triangle}{=}\>\>$} }}
\newcommand\R{\mathbb R}
\newcommand\diverg{\mathop{\mbox{\rm div}}}
\newcommand\supp{\mathop{\mbox{\rm supp}}}

\begin{document}

\baselineskip=24pt

\leftskip 0 true cm
\rightskip 0 true cm

\newpage
\begin{center}
{\large \bf On the Well-posedness of 2-D Incompressible Navier-Stokes Equations with Variable Viscosity in Critical Spaces$^*$
\footnote{$^*$This work is supported by NSFC under grant number 11571118,
 and
by the Fundamental Research Funds for the
Central Universities of China under the grant
number 2012ZZ0072.}}
\\
\bigskip
\bigskip
\footnote{$^{**}$
Corresponding Author: Huan XU}
\footnote{$^{\dag}$
Email Address:
hxuscut@163.com (H XU);\
 yshli@scut.edu.cn (YS LI);\
 pingxiaozhai@163.com (XP ZHAI). }
Huan XU$^1$, Yongsheng LI$^1$ \ and \ Xiaoping ZHAI$^2$\\[2ex]
 $^1$School of Mathematics,
 South China University of Technology,
\\[0.5ex]
 Guangzhou, Guangdong 510640, P. R. China\\[2ex]
$^2$Department of Mathematics, Sun Yat-Sen University,
\\[0.5ex]
 Guangzhou, Guangdong 510275, P. R. China

\end{center}

\bigskip
\bigskip


\centerline{\bf Abstract}
In this paper, we first prove the local well-posedness of the 2-D incompressible Navier-Stokes equations with variable viscosity in critical Besov spaces with negative regularity indices, without smallness assumption on the variation of the density. The key is to prove for $p\in(1,4)$ and $a\in\dot{B}_{p,1}^{\frac2p}(\R^2)$ that the solution mapping $\mathcal{H}_a:F\mapsto\nabla\Pi$ to the 2-D elliptic equation
$\diverg\big((1+a)\nabla\Pi\big)=\diverg F$ is bounded on $\dot{B}_{p,1}^{\frac2p-1}(\R^2)$. More precisely, we prove that
$$\|\nabla\Pi\|_{\dot{B}_{p,1}^{\frac2p-1}}\leq C\big(1+\|a\|_{\dot{B}_{p,1}^{\frac2p}}\big)^2\|F\|_{\dot{B}_{p,1}^{\frac2p-1}}.$$
The proof of the uniqueness of solution to \eqref{Model2} relies on a Lagrangian approach \cite{DM CPAM 2012,DM ARMA 2013,Danchin AIF 2014}.
When the viscosity coefficient $\mu(\rho)$ is a positive constant, we prove that \eqref{Model2} is globally well-posed.

\bigskip
\noindent {\bf Key Words:}
Incompressible Navier-Stokes equations; Littlewood-Paley theory; Lagrangian coordinates; Well-posedness.\\
\bigskip
\noindent {\bf AMS 2010 Subject Classification}:
35Q35; 76D03; 35B45.
\bigskip
\leftskip 0 true cm
\rightskip 0 true cm

\section{\large\bf Introduction}

\setcounter{Theorem}{0}
\setcounter{Lemma}{0}

\setcounter{section}{1}

\label{1} \setcounter{Theorem}{0} \setcounter{Lemma}{0}
\setcounter{section}{1}
In this paper, we study the Cauchy problem of the 2-D incompressible Navier-Stokes equations with variable viscosity in critical Besov spaces
\begin{eqnarray}\label{Model1}
\left\{\begin{aligned}
&\partial_t\rho+u\cdot\nabla\rho =0,\\
&\partial_t(\rho u)+\diverg(\rho u\otimes u)-\diverg\big(2\mu(\rho)\mathcal{M}(u)\big)+\nabla\Pi=0,\\
&\diverg u=0,\\
&(\rho,u)|_{t=0}=(\rho_0,u_0),
\end{aligned}\right.
\end{eqnarray}
where $\rho$ and $u=(u_{1},u_{2})$ stand for the density and velocity field, $\mathcal{M}(u)=\frac{1}{2}(\partial_{i}u_{j}+\partial_{j}u_{i})$, $\Pi$ is a scalar pressure function, the viscosity coefficient $\mu(\rho)$ is smooth, positive on $[0,\infty)$. Throughout, we assume that the space variable $x$ belongs to the whole space $\R^2$.

Global weak solutions with finite energy to system \eqref{Model1} were first obtained by the Russian school \cite{AKM} in the case when $\mu(\rho)=\mu>0$ and $\rho_0$ is bounded away from 0. We also refer to \cite{Lions} for an overview of results on weak solutions and to \cite{Desjardins DIE1,Desjardins DIE2,Desjardins ARMA} for some improvements. However, the uniqueness of weak solutions is not known in general. When $\mu(\rho)=\mu>0$ and $\rho_0$ is bounded away from 0, Ladyzhenskaya and Solonnikov \cite{Ladyzhenskaya Solonnikov} initiated the studies for unique solvability of system \eqref{Model1} in a bounded domain $\Omega$ with homogeneous Dirichlet boundary condition for $u$. Similar results were established by Danchin \cite{Danchin ADE 2004} in the whole space $\R^n$ with initial data in the almost critical Sobolev spaces. On the other hand, from the viewpoint of physics, it is interesting to study the case for which density is discontinuous. Recently, Danchin and Mucha \cite{DM ARMA 2013} proved by using a Lagrangian approach that the system \eqref{Model1} has a unique local solution with initial data $(\rho_0, u_0)\in L^{\infty}(\R^n)\times H^2(\R^n)$ if initial vacuum dose not occur, see also some improvements in \cite{HPZ JMPA 2013,PZZ CPDE 2013}.

On the other hand, if the density $\rho$ is away from zero, we denote by $a\stackrel{\mathrm{def}}{=}\frac{1}{\rho}-1$ and
$\tilde{\mu}(a)\stackrel{\mathrm{def}}{=}\mu(\frac{1}{1+a})$ so that the system \eqref{Model1} can be equivalently reformulated as
\begin{eqnarray}\label{Model2}
\left\{\begin{aligned}
&\partial_t a+u\cdot\nabla a =0,\\
&\partial_t u+u\cdot\nabla u-(1+a)\left\{\diverg\big(2\tilde{\mu}(a)\mathcal{M}(u)\big)-\nabla\Pi\right\}=0,\\
&\diverg u=0,\\
&(a,u)|_{t=0}=(a_0,u_0).
\end{aligned}\right.
\end{eqnarray}
Just as the classical Navier-Stokes equations, the system \eqref{Model2} also has a scaling. Indeed, if $(a,u)$ solves \eqref{Model2} with initial data $(a_0,u_0)$, then for any $\lambda>0$,
$$
(a,u)_\lambda(t,x)\stackrel{\mathrm{def}}{=}(a(\lambda^2t,\lambda x),\lambda u(\lambda^2t,\lambda x))
$$
also solves \eqref{Model2} with initial data $(a_0(\lambda\cdot),\lambda u_0(\lambda\cdot))$. Moreover, the norm of $(a_0(\lambda\cdot),\lambda u_0(\lambda\cdot))$ is independent of $\lambda$ in the so called critical spaces $\dot{B}_{p,1}^{\frac{2}{p}}(\R^{2})\times\dot{B}_{p,1}^{\frac{2}{p}-1}(\R^{2})$. In resent ten years, the French school studied the well-posedness of incompressible or compressible fluids in the framework of critical Besov spaces (see, e.g., \cite{Abidi RMI 2007,AP AIF 2007,Danchin CPDE 2001,Danchin PRSES 2003,Danchin CPDE 2007}).

Motivated by \cite{AGZ ARMA 2012,AGZ JMPA 2013,Danchin CPDE 2007,Danchin JDE 2010,Danchin AIF 2014} concerning the well-posedness of the incompressible or compressible fluids without smallness assumption on the variation of the initial density, we study the well-posedness of \eqref{Model2} in critical Besov spaces with negative regularity. More precisely, we will prove the following main theorem:
\begin{Theorem}\label{Theorem1.1}
Let $p\in(1,4)$, $(a_0,u_0)\in\dot{B}_{p,1}^{\frac{2}{p}}(\R^2)\times\dot{B}_{p,1}^{\frac{2}{p}-1}(\R^2)$
with $\diverg u_0=0$ and $1+a_0\geq\kappa>0$. Assume that $\tilde{\mu}(a)$
is a smooth, positive function on $[0,\infty)$. Then \eqref{Model2} has a unique local solution $(a,u,\nabla\Pi)$ on $[0,T]$ such that
\begin{align}\label{soln regularity}
a\in& C([0,T];\dot{B}_{p,1}^{\frac{2}{p}}(\R^2))\cap \widetilde{L}_T^\infty(\dot{B}_{p,1}^{\frac{2}{p}}(\R^2)),
\ \nabla\Pi\in L_T^1(\dot{B}_{p,1}^{\frac{2}{p}-1}(\R^2)),\nonumber\\
u\in& C([0,T];\dot{B}_{p,1}^{\frac{2}{p}-1}(\R^2))\cap \widetilde{L}_T^\infty(\dot{B}_{p,1}^{\frac{2}{p}-1}(\R^2))
\cap L_T^1(\dot{B}_{p,1}^{\frac{2}{p}+1}(\R^2)).
\end{align}
Moreover, if $\tilde{\mu}(a)$ is a positive constant and $a_0\in B_{p,1}^{\frac{2}{p}}(\R^2)$,
then \eqref{Model2} has a unique global solution $(a,u,\nabla\Pi)$ such that for any $t>0$\rm{:}
\begin{align}\label{global soln}
\|a\|_{\widetilde{L}_t^\infty(B_{p,1}^{\frac{2}{p}})}+\|u\|_{\widetilde{L}_t^\infty(\dot{B}_{p,1}^{\frac{2}{p}-1})}+\|u\|_{L_t^1(\dot{B}_{p,1}^{\frac{2}{p}+1})}
+\|\nabla\Pi\|_{L_t^1(\dot{B}_{p,1}^{\frac{2}{p}-1})}\leq C\exp\left\{C\exp\big(Ct^\frac12\big)\right\}.
\end{align}
for some time independent constant $C$.
\end{Theorem}
As in \cite{AGZ ARMA 2012,AGZ JMPA 2013,Danchin AFSTM 2006,Danchin JDE 2010}, a central problem to prove the local well-posedness part of Theorem \ref{Theorem1.1} is the estimate of the pressure function. However, in the particular case when the spatial dimension is two, we could prove for $p\in(1,4)$ and $a\in\dot{B}_{p,1}^{\frac2p}(\R^2)$ that the solution mapping $\mathcal{H}_a:F\mapsto\nabla\Pi$ to the 2-D elliptic equation
\begin{align}\label{elliptic equation}
\diverg\big((1+a)\nabla\Pi\big)=\diverg F\quad {\rm in}\quad\R^2
\end{align}
is bounded on $\dot{B}_{p,1}^{\frac2p-1}(\R^2)$ (see Proposition \ref{Proposition3.1} below). This in some sense explains that the 2-D problem is a critical problem. We shall also mention that the proof of the uniqueness of solution to \eqref{Model2} relies on a Lagrangian approach \cite{DM CPAM 2012,DM ARMA 2013,Danchin AIF 2014}. Finally, if the viscosity coefficient $\mu(\rho)=\widetilde{\mu}(a)\equiv1$, then the systems \eqref{Model1} and \eqref{Model2} turn respectively into
\begin{eqnarray}\label{Modelrho}
\left\{\begin{aligned}
&\partial_t\rho+u\cdot\nabla\rho=0,\\
&\rho(\partial_t u+u\cdot\nabla u)-\Delta u+\nabla\Pi=0,\\
&\diverg u=0,\\
&(\rho,u)|_{t=0}=(\rho_0,u_0),
\end{aligned}\right.
\end{eqnarray}
and
\begin{eqnarray}\label{Modela}
\left\{\begin{aligned}
&\partial_t a+u\cdot\nabla a =0,\\
&\partial_t u+u\cdot\nabla u-(1+a)\Delta u+(1+a)\nabla\Pi=0,\\
&\diverg u=0,\\
&(a,u)|_{t=0}=(a_0,u_0).
\end{aligned}\right.
\end{eqnarray}
Notice that $u_0\in\dot{B}_{p,1}^{\frac{2}{p}-1}(\R^2)$ is not of finite energy for $p\in(2,4)$. For this, we set $\bar{u}\stackrel{\mathrm{def}}{=}u-u_F$ with $u_F(t)\stackrel{\mathrm{def}}{=}e^{(t-t_1)\Delta}u(t_1)$. Then we deduce from \eqref{Modelrho} that $(\rho,\bar{u},\nabla\Pi)$ solves
\begin{eqnarray}\label{Modelbar}
\left\{\begin{aligned}
&\partial_t\rho+\diverg(\rho(u_F+\bar{u}))=0,\\
&\rho\big(\partial_t\bar{u}+(u_F+\bar{u})\cdot\nabla\bar{u}\big)-\Delta\bar{u}+\nabla\Pi=G,\\
&\diverg \bar{u}=0,\\
&\rho|_{t=t_1}=\rho(t_1),\quad\bar{u}|_{t=t_1}=0
\end{aligned}\right.
\end{eqnarray}
with $G\stackrel{\rm{def}}{=}(1-\rho)\Delta u_F-\rho u_F\cdot\nabla u_F-\rho\bar{u}\cdot\nabla u_F$. By using the method in \cite{AGZ ARMA 2012,AGZ JMPA 2013}, we could present the $L^2$ energy estimate for $\bar{u}$ to prove the global well-posedness part of Theorem \ref{Theorem1.1}.

The remainder of this paper is organized as follows. In Section 2, we present some basic facts on Littlewood--Paley analysis and introduce several technical lemmas, then we present the estimates to the free transport equation. In Section 3, we study some linear elliptic and parabolic equations with rough coefficients in the framework of critical Besov spaces. In Section 4, we complete the proof of the local well-posedness part of Theorem \ref{Theorem1.1}. In Section 5, we present the energy estimate for $\bar{u}$ in the $L^2$ framework to prove the global well-posedness of \eqref{Model2} with $\mu(\rho)\equiv1$.

\bigskip

\noindent{\rm\bf Notations:}
For two operators $A$ and $B$, we denote $[A,B]=AB-BA$ the commutator between $A$ and $B$. The letter $C$ stands for a generic constant whose meaning is clear from the context. We sometimes write $a\lesssim b$ instead of $a\leq Cb$. For $p\in[1,\infty]$, the conjugate index $p'$ is determined by $\frac1p+\frac{1}{p'}=1$. The Fourier transform of $u$ is denoted either by $\hat{u}$ or $\mathscr{F}u$, the inverse by $\mathscr{F}^{-1}u$. The notation $\mathcal{P}$ stands for the Leray projector on the divergence free vector fields, while $\mathcal{Q}=\rm{Id}-\mathcal{P}$ stands for the projector on the gradient type vector fields.

For $X$ a Banach space and $I$ an interval of $\mathbb{R}$, we denote by $C(I;X)$ the set of continuous functions on $I$ with values in $X$.
For $q\in [1,+\infty]$, $L^q(I;X)$ stands for the set of measurable functions on $I$ with values in $X$, such that $t\mapsto\|f(t)\|_{X}$ belongs to $L^q(I)$.
For short, we sometimes write $L_T^q(X)$ instead of $L^q((0,T);X)$.

\section{\large\bf Preliminaries}
We first recall some basic facts on Littlewood-Paley theory (see \cite{BCD} for instance).
Let $\chi,\varphi$ be two smooth radial functions valued in the interval [0,1],
the support of $\chi$ be the ball $\mathscr{B}=\{\xi\in\R^2: |\xi|\leq\frac{4}{3}\}$
while the support of $\varphi$ be the annulus $\mathscr{C}=\{\xi\in\R^2: \frac{3}{4}\leq|\xi|\leq\frac{8}{3}\}$, and satisfy
\begin{align*}
\sum_{j\in\mathbb{Z}}\varphi(2^{-j}\xi)=&1\quad  \rm{for} \quad \xi\in\R^2\setminus\{0\};\\
\chi(\xi)+\sum_{j\geq 0}\varphi(2^{-j}\xi)=&1\quad  \rm{for} \quad \xi\in\R^2.
\end{align*}
Denote by $h\stackrel{\mathrm{def}}{=}\mathscr{F}^{-1}\varphi$ and $\widetilde{h}\stackrel{\mathrm{def}}{=}\mathscr{F}^{-1}\chi$,
the homogeneous dyadic blocks $\dot{\Delta}_{j}$ and the homogeneous low-frequency cutoff operators
$\dot{S}_{j}$ are defined for all $j\in\mathbb{Z}$ by
\begin{align*}
\dot{\Delta}_{j}u=&\varphi(2^{-j}D)u=2^{2j}\int_{\mathbb{R}^2}h(2^{j}y)u(x-y)dy,\\
\dot{S}_{j}u=&\chi(2^{-j}D)u=2^{2j}\int_{\mathbb{R}^2}\widetilde{h}(2^{j}y)u(x-y)dy.
\end{align*}
Denote by $\mathscr{S}_h^{'}(\R^2)$ the space of tempered distributions $u$ such that
$$
\lim_{j\rightarrow -\infty}\dot{S}_j u=0\ \ in\ \ \mathscr{S}^{'}(\R^2).
$$
Then we have the formal decomposition
$$
u=\sum_{ j\in \mathbb{Z} }\dot{ \Delta }_{j}u,\ \ \  \forall u\in\mathscr{S}_h^{'}(\mathbb{R}^2).
$$
Moreover, the Littlewood-Paley decomposition satisfies the property of almost orthogonality:
$$
\dot{\Delta}_{k}\dot{\Delta}_{j}u\equiv0\ \ \mathrm{if}\ \ |k-j|\geq2,\ \
\mathrm{and}\ \ \dot{\Delta}_{k}(\dot{S}_{j-1}u\dot{\Delta}_{j}v)\equiv 0\ \ \mathrm{if}\ \ |k-j|\geq5.
$$

Now we recall the definition of homogeneous Besov spaces from \cite{BCD}.
\begin{Definition}
Let $s\in\R$ and $1\leq p,r\leq\infty$. The homogeneous Besov space $\dot{B}_{p,r}^{s}(\R^2)$
consists of all the distributions $u$ in $\mathscr{S}_{h}^{'}(\R^2)$ such that
$$
\|u\|_{\dot{B}_{p,r}^{s}}\stackrel{\mathrm{def}}{=}\left\|\big(2^{js}\|\dot{\Delta}_{j}u\|_{L^{p}}\big)_{j\in\mathbb{Z}}\right\|_{l^{r}}
<\infty.
$$
\begin{Remark}
With some slight modifications, we can also define inhomogeneous Besov spaces.
Indeed, for $u\in\mathscr{S}'(\R^2)$, we define
\begin{align*}
\Delta_ju=\dot{\Delta}_{j}u,\ \forall j\geq0;\ \ \ \Delta_{-1}u=&\dot{S}_{0}u;\ \ \ \Delta_ju=0,\ \forall j\leq -2;\\
\mathrm{and}\ \ S_{j}u=&\sum_{j'\leq j-1}\Delta_{j'}u.
\end{align*}
Then the inhomogeneous Besov space $B_{p,r}^{s}(\R^2)$
consists of all the distributions $u$ in $\mathscr{S}^{'}(\R^2)$ such that
$$
\|u\|_{B_{p,r}^{s}}\stackrel{\mathrm{def}}{=}\left\|\big(2^{js}\|\Delta_{j}u\|_{L^{p}}\big)_{j\geq-1}\right\|_{l^{r}}
<\infty.
$$
\end{Remark}
\begin{Remark}
Let $s\in\mathbb{R}$ and $1\leq p,r\leq\infty$. Then there exists a positive constant $C$ such that $u$ belongs to $\dot{B}_{p,r}^{s}(\mathbb{R}^2)$ if and only if there exists $\{c_{j,r}\}_{j\in\mathbb{Z}}$ such that $c_{j,r}\geq0$, $\|c_{j,r}\|_{l^{r}}=1$ and
$$
\|\dot{\Delta}_{j}u\|_{L^p}\leq Cc_{j,r}2^{-js}\|u\|_{\dot{B}_{p,r}^{s}},\quad \forall j\in\mathbb{Z}.
$$
For simplicity, we denote $d_j\stackrel{\mathrm{def}}{=}c_{j,1}$ and $c_j\stackrel{\mathrm{def}}{=}c_{j,2}$.
\end{Remark}
\end{Definition}

To gain a better description of the regularization effect to the transport-diffusion equation,
we should use the Chemin-Lerner type norms(see \cite{BCD}):
\begin{Definition}
Let $s\in\mathbb{R}$ and $0<T\leq +\infty$. We define
$$
\|u\|_{\widetilde{L}_{T}^{\sigma}(\dot{B}_{p,r}^{s})}\stackrel{\mathrm{def}}{=}
\left(
\sum_{j\in\mathbb{Z}}
2^{jrs}\left(\int_{0}^{T}\|\dot{\Delta}_{j}u(t)\|_{L^p}^{\sigma}dt\right)^{\frac{r}{\sigma}}
\right)^{\frac{1}{r}}
$$
for $p\in[1,\infty]$, $r, \sigma \in [1,\infty)$,
and with the standard modification for $r=\infty$ or $\sigma=\infty$.
\end{Definition}

The following lemmas will be repeatedly used in this paper (see \cite{BCD}).
\begin{Lemma}\label{Bernstein Lemma}
Let $\mathscr{C}\subset\R^2$ be an annulus and $\mathscr{B}\subset\R^2$ be a ball.
There exists a positive constant $C$ such that for any $0\leq k\in\mathbb{Z}$,
any $\lambda>0$, any smooth homogeneous function $\sigma$ of degree m, any $1\leq p\leq q\leq\infty$, and any function $u\in L^{p}$, we have
\begin{align*}
\supp\widehat{u}\subset\lambda\mathscr{B}\Longrightarrow&
\|D^{k}u\|_{L^{q}}\stackrel{\mathrm{def}}{=}\sum_{|\alpha|=k}\|\partial^{\alpha}u\|_{L^{q}}\leq C^{k+1}\lambda^{k+2(\frac{1}{p}-\frac{1}{q})}\|u\|_{L^{p}},\\
\supp\widehat{u}\subset\lambda\mathscr{C}\Longrightarrow&
C^{-k-1}\lambda^{k}\|u\|_{L^{p}}\leq\|D^{k}u\|_{L^{p}}\leq C^{k+1}\lambda^{k}\|u\|_{L^{p}},\\
\supp\widehat{u}\subset\lambda\mathscr{C}\Longrightarrow&\|\sigma(D)u\|_{L^q}\le C_{\sigma,m}\lambda^{m+2(\frac1p-\frac1q)}\|u\|_{L^p}.
\end{align*}
\end{Lemma}
\begin{Lemma}\label{Heatflow}
Let $\mathscr{C}\subset\R^2$ be an annulus. Then there exist positive constants $c$ and $C$, such that for any $1\leq p\leq\infty$ and $\lambda>0$, we have
$$\supp\hat{u}\subset \lambda\mathscr{C}\Rightarrow\|e^{t\Delta}u\|_{L^p}\leq Ce^{-ct\lambda^2}\|u\|_{L^p}.$$
\end{Lemma}

On the other hand, it has been demonstrated that the Bony's decomposition \cite{BCD,Bony} is
very effective to deal with nonlinear problems.
Here, we recall the Bony's decomposition in the homogeneous context:
$$
uv=\dot{T}_{u}v+\dot{T}_{v}u+\dot{R}(u,v)=\dot{T}_{u}v+\dot{T}'_{v}u,
$$
where
\begin{align*}
\dot{T}_{u}v\stackrel{\mathrm{def}}{=}\sum_{j\in\mathbb{Z}}\dot{S}_{j-1}u\dot{\Delta}_{j}v ,\ \ &\dot{R}(u,v)\stackrel{\mathrm{def}}{=}\sum_{j\in\mathbb{Z}}\dot{\Delta}_{j}u\widetilde{\dot{\Delta}}_{j}v \ \ \mathrm{with}\ \ \widetilde{\dot{\Delta}}_{j}v\stackrel{\mathrm{def}}{=}\sum_{|j^{'}-j|\leq1}\dot{\Delta}_{j^{'}}v,\\
\mathrm{and}\ \ &\dot{T}'_{v}u\stackrel{\mathrm{def}}{=}\dot{T}_{v}u+\dot{R}(u,v)=\sum_{j\in\mathbb{Z}}\dot{\Delta}_{j}u\dot{S}_{j+2}v.
\end{align*}

In the sequel, we should frequently use the following product laws \cite{XLC}.
\begin{Lemma}\label{Product Laws}
Let $1\leq p, q\leq \infty$, $s_1\leq \frac{2}{q}$, $s_2\leq 2\min\{\frac1p,\frac1q\}$ and $s_1+s_2>2\max\{0,\frac1p +\frac1q -1\}$.
Then there holds
\begin{align*}
\|ab\|_{\dot{B}_{p,1}^{s_1+s_2 -\frac{2}{q}}}\lesssim \|a\|_{\dot{B}_{q,1}^{s_1}}\|b\|_{\dot{B}_{p,1}^{s_2}},\ \
\forall (a,b)\in\dot{B}_{q,1}^{s_1}(\R^2)\times\dot{B}_{p,1}^{s_2}(\R^2).
\end{align*}
\begin{Remark} We shall frequently use the fact that $\dot{B}_{p,1}^{\frac2p}(\R^2)$ is an algebra for $p\in[1,\infty)$ and that $\|ab\|_{\dot{B}_{p,1}^{\frac{2}{p}-1}}\lesssim \|a\|_{\dot{B}_{p,1}^{\frac{2}{p}}}\|b\|_{\dot{B}_{p,1}^{\frac{2}{p}-1}}$ for $p\in[1,4)$.
\end{Remark}
\end{Lemma}

Let us also recall the following commutator estimates (see \cite[Lemma 2.100]{BCD} for instance).
\begin{Lemma}\label{Commutator Estimates1}
Let $(p,q)\in[1,\infty]^2$, $-1-2\min\{\frac1p,\frac{1}{q'}\}<s\leq 1+2\min\{\frac1p,\frac1q\}$,
$a\in \dot{B}_{q,1}^{s}(\R^2)$ and $u\in \dot{B}_{p,1}^{\frac{2}{p}+1}(\R^2)$ with $\diverg u=0$.
Then there holds
$$
\|[u\cdot \nabla ,\dot{\Delta}_j]a\|_{L^q}\lesssim d_j 2^{-js}\|u\|_{\dot{B}_{p,1}^{\frac{2}{p}+1}}\|a\|_{\dot{B}_{q,1}^{s}}.
$$
\end{Lemma}

Motivated by \cite{AGZ ARMA 2012,AGZ JMPA 2013,Danchin CPDE 2007}, we need the following proposition to deal with the transport equation in \eqref{Model2}.
\begin{Proposition}\label{Proposition2.1}
Let $(p,q)\in[1,\infty]^2$, $\frac{1}{q}-\frac1p\leq\frac12$,
$a_{0}\in\dot{B}_{q,1}^{\frac{2}{q}}(\R^2)$ and $u\in L_{T}^{1}(\dot{B}_{p,1}^{\frac{2}{p}+1}(\R^2))$ with $\diverg u=0$.
If $a\in C([0,T];\dot{B}_{q,1}^{\frac{2}{q}}(\R^2))$ solves
\begin{align*}
\partial_{t}a+u\cdot\nabla a=0,\quad(t,x)\in(0,T]\times\R^{2}
\end{align*}
with initial data $a_{0}$, then there holds for $t\in[0, T]$
\begin{align}\label{transport estimate1}
\|a\|_{\widetilde{L}_t^{\infty}(\dot{B}_{q,1}^{\frac{2}{q}})}\leq\|a_0\|_{\dot{B}_{q,1}^{\frac{2}{q}}}e^{CU(t)},
\end{align}
and
\begin{align}\label{transport estimate2}
\|a-\dot{S}_ma\|_{\widetilde{L}_t^{\infty}(\dot{B}_{q,1}^{\frac{2}{q}})}
\leq\sum_{j\geq m}2^{\frac{2j}{q}}\|\dot{\Delta}_{j}a_0\|_{L^q}+\|a_0\|_{\dot{B}_{q,1}^{\frac{2}{q}}}\big(e^{CU(t)}-1\big),
\end{align}
with $U(t)\stackrel{\rm{def}}{=}\|u\|_{L_t^1(\dot{B}_{p,1}^{\frac{2}{p}+1})}$.
\end{Proposition}
\begin{proof} We first get by applying $\dot{\Delta}_{j}$ to the transport equation
$$
(\partial_{t}+u\cdot\nabla)\dot{\Delta}_{j}a=[u\cdot\nabla ,\dot{\Delta}_{j}]a.
$$
Since $\nabla u\in L_T^1(L^\infty)$ with $\diverg u=0$, we get by using classical $L^q$ estimate for transport equation that
\begin{align*}
\|\dot{\Delta}_{j}a\|_{L_{t}^{\infty}(L^q)}\leq\|\dot{\Delta}_{j}a_0\|_{L^q}
+\int_0^t\|[u\cdot\nabla ,\dot{\Delta}_{j}]a\|_{L^{q}}d\tau,
\end{align*}
from which and Lemma \ref{Commutator Estimates1}, we get for $\frac{1}{q}-\frac1p\leq\frac12$
\begin{align}\label{teq1}
2^{\frac{2j}{q}}\|\dot{\Delta}_{j}a\|_{L_{t}^{\infty}(L^q)}\leq2^{\frac{2j}{q}}\|\dot{\Delta}_{j}a_0\|_{L^q}
+C\int_0^t d_{j}(\tau)\|u(\tau)\|_{\dot{B}_{p,1}^{\frac{2}{p}+1}} \|a(\tau)\|_{\dot{B}_{q,1}^{\frac{2}{q}}}d\tau.
\end{align}
Taking summation for $j\in\mathbb{Z}$ and then using Gronwall's inequality implies \eqref{transport estimate1}.
Then substituting \eqref{transport estimate1} into \eqref{teq1} results in
\begin{align*}
2^{\frac{2j}{q}}\|\dot{\Delta}_{j}a\|_{L_{t}^{\infty}(L^q)}\leq2^{\frac{2j}{q}}\|\dot{\Delta}_{j}a_0\|_{L^q}
+C\|a_0\|_{\dot{B}_{q,1}^{\frac{2}{q}}}\int_0^t d_{j}(\tau)U'(\tau)e^{CU(\tau)}d\tau.
\end{align*}
Summing up the above inequality on $\{j\geq m\}$ leads to \eqref{transport estimate2}.
\end{proof}
\begin{Remark} Let $p,q,U(t)$ be determined by Proposition \ref{Proposition2.1}.
Then in the framework of inhomogeneous Besov spaces, there similarly holds
\begin{align*}
\|a\|_{\widetilde{L}_t^{\infty}(B_{q,1}^{\frac{2}{q}})}\leq&\|a_0\|_{B_{q,1}^{\frac{2}{q}}}e^{CU(t)},\\
\|a-S_ma\|_{\widetilde{L}_t^{\infty}(B_{q,1}^{\frac{2}{q}})}
\leq&\sum_{j\geq m}2^{\frac{2j}{q}}\|\Delta_{j}a_0\|_{L^q}+\|a_0\|_{B_{q,1}^{\frac{2}{q}}}\big(e^{CU(t)}-1\big).
\end{align*}
\end{Remark}

\section{\bf Linear system with rough coefficients}
In this section, we study some linear elliptic and parabolic equations with rough coefficients in the $L^p$ framework. We need the following commutator estimates of integral type:
\begin{Lemma}\label{Lemma3.1}
\rm{(\romannumeral1)} Let $(p,q)\in(\frac{1+\sqrt{17}}{4},2]\times[1,\infty)$ with $\frac1p-\frac1q\leq\frac12$.
Then we have
\begin{align}\label{Ijpq}
I_j\stackrel{\rm{def}}{=}\int_{\R^2}\diverg\big([\dot{\Delta}_j,a]\nabla\Pi\big)\cdot|\dot{\Delta}_j\Pi|^{p-2}\dot{\Delta}_j\Pi dx
\lesssim d_j2^{j(2-\frac{2}{p})}\|a\|_{\dot{B}_{q,1}^{\frac{2}{q}}}\|\nabla\Pi\|_{L^2}\|\dot{\Delta}_j\Pi\|_{L^p}^{p-1}.
\end{align}

\rm{(\romannumeral2)} For $p\in(1,4)$, we alternatively have
\begin{align}
I_j\lesssim& d_j2^{j(2-\frac{2}{p})}\|a\|_{\dot{B}_{p,1}^{\frac{2}{p}}}\|\nabla\Pi\|_{L^2}\|\dot{\Delta}_j\Pi\|_{L^p}^{p-1},\ \ p\in(1,2),\label{Ijp12}\\
I_j\lesssim& d_j2^{j(2-\frac{2}{p})}\|a\|_{\dot{B}_{p,1}^{\frac{2}{p}}}\|\nabla\Pi\|_{\dot{B}_{p,2}^{\frac{2}{p}-1}}
\|\dot{\Delta}_j\Pi\|_{L^p}^{p-1},\ \ p\in[2,4).\label{Ijp24}
\end{align}
\end{Lemma}
\begin{proof} \rm{(\romannumeral1)} Note that we could not directly use integration by parts. For this, we first get by using Bony's decomposition
\begin{align*}
I_j=&\int_{\R^2}\diverg\big([\dot{\Delta}_j,\dot{T}_{a}]\nabla\Pi\big)\cdot|\dot{\Delta}_j\Pi|^{p-2}\dot{\Delta}_j\Pi dx
+\int_{\R^2}\diverg\dot{\Delta}_j\big(\dot{T}'_{\nabla\Pi}a\big)\cdot|\dot{\Delta}_j\Pi|^{p-2}\dot{\Delta}_j\Pi dx\\
&-\int_{\R^2}\diverg\big(\dot{T}'_{\dot{\Delta}_j\nabla\Pi}a\big)\cdot|\dot{\Delta}_j\Pi|^{p-2}\dot{\Delta}_j\Pi dx
\stackrel{\rm{def}}{=}I_j^1+I_j^2+I_j^3.
\end{align*}
By the definition of Bony's decomposition, we have
\begin{align}\label{cmt}
[\dot{\Delta}_j,\dot{T}_{a}]\nabla\Pi
=-2^{2j}\sum_{|j'-j|\leq4}\int_{\R^2}h(2^jy)\dot{\Delta}_{j'}\nabla\Pi(x-y)dy\int_0^1y\cdot\nabla\dot{S}_{j'-1}a(x-\tau y)d\tau,
\end{align}
from which, we get by using H$\rm{\ddot{o}}$lder inequality and Lemma \ref{Bernstein Lemma} that
\begin{align*}
\|[\dot{\Delta}_j,\dot{T}_{a}]\nabla\Pi\|_{L^p}
\lesssim2^{-j}\sum_{|j'-j|\leq4}\|\nabla\dot{S}_{j'-1}a\|_{L^{\frac{2p}{2-p}}}\|\dot{\Delta}_{j'}\nabla\Pi\|_{L^2}
\lesssim d_j2^{j(1-\frac{2}{p})}\|a\|_{\dot{B}_{q,1}^{\frac{2}{q}}}\|\nabla\Pi\|_{L^2},
\end{align*}
where we used $\|\nabla\dot{S}_{j'-1}a\|_{L^{\frac{2p}{2-p}}}\lesssim d_{j'}2^{j'(2-\frac2p)}\|a\|_{\dot{B}_{q,1}^{\frac{2}{q}}}$
for $p>1$ and $\frac1p-\frac1q\leq\frac12$.
Note that $[\dot{\Delta}_j,\dot{T}_{a}]\nabla\Pi$ is spectrally supported in an annulus of size $2^j$. Whence we infer
\begin{align}\label{Ij1}
I_j^1\lesssim d_j2^{j(2-\frac{2}{p})}\|a\|_{\dot{B}_{q,1}^{\frac{2}{q}}}\|\nabla\Pi\|_{L^2}\|\dot{\Delta}_j\Pi\|_{L^p}^{p-1}.
\end{align}
Owing to the localization properties of the Littlewood-Paley decomposition, we have
\begin{align*}
\dot{\Delta}_j\big(\dot{T}'_{\nabla\Pi}a\big)=\sum_{j'\geq j-3}\dot{\Delta}_j\big(\dot{\Delta}_{j'}a\dot{S}_{j'+2}\nabla\Pi\big).
\end{align*}
If $q\geq2$, we denote $\frac1\gamma\stackrel{\rm{def}}{=}\frac12+\frac1q\geq\frac1p$ and apply Lemma \ref{Bernstein Lemma} to obtain
\begin{align*}
\|\dot{\Delta}_j\big(\dot{T}'_{\nabla\Pi}a\big)\|_{L^p}\lesssim2^{2j(\frac1\gamma-\frac1p)}\sum_{j'\geq j-3}\|\dot{\Delta}_{j'}a\|_{L^{q}}\|\dot{S}_{j'+2}\nabla\Pi\|_{L^2}
\lesssim d_j2^{j(1-\frac{2}{p})}\|a\|_{\dot{B}_{q,1}^{\frac{2}{q}}}\|\nabla\Pi\|_{L^2}.
\end{align*}
While if $q<2$, the embedding $\dot{B}_{q,1}^{\frac{2}{q}}(\R^2)\hookrightarrow\dot{B}_{2,1}^{1}(\R^2)$ ensures that the above inequality still holds.
Thus we obtain
\begin{align}\label{Ij2}
I_j^2\lesssim d_j2^{j(2-\frac{2}{p})}\|a\|_{\dot{B}_{q,1}^{\frac{2}{q}}}\|\nabla\Pi\|_{L^2}\|\dot{\Delta}_j\Pi\|_{L^p}^{p-1}.
\end{align}
For $I_j^3$, we write
\begin{align*}
\dot{T}'_{\dot{\Delta}_j\nabla\Pi}a=\sum_{j'-j=-1,-2}\dot{\Delta}_{j'}a\dot{S}_{j'+2}\dot{\Delta}_{j}\nabla\Pi
+\sum_{j'\geq j}\dot{\Delta}_{j'}a\dot{\Delta}_{j}\nabla\Pi,
\end{align*}
from which, we get by applying Lemma 8 in Appendix B of \cite{Danchin JDE 2010}
\begin{align*}
I_j^3=&-\sum_{j'-j=-1,-2}
\int_{\R^2}\diverg\big(\dot{\Delta}_{j'}a\dot{S}_{j'+2}\dot{\Delta}_{j}\nabla\Pi\big)\cdot|\dot{\Delta}_j\Pi|^{p-2}\dot{\Delta}_j\Pi dx\\
&+(p-1)\sum_{j'\geq j}\int_{\R^2}\dot{\Delta}_{j'}a|\dot{\Delta}_{j}\nabla\Pi|^2\cdot|\dot{\Delta}_j\Pi|^{p-2}dx
\stackrel{\rm{def}}{=}I_j^{3,1}+I_j^{3,2}.
\end{align*}
Then it is easy to observe for $\frac1p-\frac1q\leq\frac12$ that
\begin{align}
I_j^{3,1}
\lesssim&2^{j}\sum_{j'-j=-1,-2}\|\dot{\Delta}_{j'}a\|_{L^{\frac{2p}{2-p}}}\|\dot{\Delta}_{j}\nabla\Pi\|_{L^2}\|\dot{\Delta}_j\Pi\|_{L^p}^{p-1}\nonumber\\
\lesssim&d_j2^{j(2-\frac{2}{p})}\|a\|_{\dot{B}_{q,1}^{\frac{2}{q}}}\|\nabla\Pi\|_{L^2}\|\dot{\Delta}_j\Pi\|_{L^p}^{p-1}.
\end{align}
While the assumption $p\in(\frac{1+\sqrt{17}}{4},2]$ ensures that $\frac{1}{p-1}<\frac{2p}{2-p}$.
In the case when $\max\{p,\frac{1}{p-1}\}<q\leq\frac{2p}{2-p}$, we have $(p-2)q'+1>0$ so that we could use a similar approximate argument
as the proof of Lemma A.5 in the appendix of \cite{Danchin CPDE 2001} to obtain
\begin{align*}
&\big\||\dot{\Delta}_{j}\nabla\Pi|^2\cdot|\dot{\Delta}_j\Pi|^{p-2}\big\|_{L^{q'}}^{q'}\nonumber\\
=&\int_{\R^2}|\dot{\Delta}_j\Pi|^{(p-2)q'}\dot{\Delta}_{j}\nabla\Pi\cdot\dot{\Delta}_{j}\nabla\Pi|\dot{\Delta}_{j}\nabla\Pi|^{2q'-2}dx\nonumber\\
=&-\frac{1}{(p-2)q'+1}\int_{\R^2}|\dot{\Delta}_j\Pi|^{(p-2)q'}\dot{\Delta}_{j}\Pi
\cdot\diverg\big(\dot{\Delta}_{j}\nabla\Pi|\dot{\Delta}_{j}\nabla\Pi|^{2q'-2}\big)dx.
\end{align*}
Denoting $\frac1r\stackrel{\rm{def}}{=}\frac1p-\frac1q\leq\frac12$ and using H$\rm{\ddot{o}}$lder inequality and Lemma \ref{Bernstein Lemma} gives
\begin{align*}
&\big\||\dot{\Delta}_{j}\nabla\Pi|^2\cdot|\dot{\Delta}_j\Pi|^{p-2}\big\|_{L^{q'}}^{q'}\nonumber\\
\lesssim&\big\||\dot{\Delta}_j\Pi|^{(p-2)q'+1}\big\|_{L^{\frac{p}{(p-2)q'+1}}}
\big\||\dot{\Delta}_{j}\nabla\Pi|^{q'-1}\big\|_{L^{\frac{p}{q'-1}}}\big\||\dot{\Delta}_{j}\nabla\Pi|^{q'-1}\big\|_{L^{\frac{r}{q'-1}}}
\big\|\nabla^2\dot{\Delta}_{j}\Pi\big\|_{L^{r}}\nonumber\\
\lesssim&2^{jq'(2-\frac2p+\frac2q)}\big\|\dot{\Delta}_j\Pi\big\|_{L^{p}}^{(p-1)q'}\big\|\nabla\dot{\Delta}_{j}\Pi\big\|_{L^{2}}^{q'},
\end{align*}
which implies
\begin{align}\label{Ij32}
I_j^{3,2}\lesssim&\sum_{j'\geq j}\|\dot{\Delta}_{j'}a\|_{L^{q}}
\big\||\dot{\Delta}_{j}\nabla\Pi|^2\cdot|\dot{\Delta}_j\Pi|^{p-2}\big\|_{L^{q'}}\nonumber\\
\lesssim&d_j2^{j(2-\frac{2}{p})}\|a\|_{\dot{B}_{q,1}^{\frac{2}{q}}}\|\nabla\Pi\|_{L^2}\big\|\dot{\Delta}_j\Pi\big\|_{L^{p}}^{p-1}.
\end{align}
Similarly, \eqref{Ij32} is valid for $q\leq\max\{p,\frac{1}{p-1}\}$ according to embedding.
Summing up the inequalities \eqref{Ij1}--\eqref{Ij32} results in \eqref{Ijpq}.

\rm{(\romannumeral2)} We first consider the case when $p\in(1,2)$. From \eqref{Ij1} and \eqref{Ij2}, we infer
\begin{align}\label{Ij12}
I_j^{1}+I_j^{2}\lesssim d_j2^{j(2-\frac{2}{p})}\|a\|_{\dot{B}_{p,1}^{\frac{2}{p}}}\|\nabla\Pi\|_{L^2}\|\dot{\Delta}_j\Pi\|_{L^p}^{p-1}.
\end{align}
For $I_j^3$, we alternatively get
\begin{align*}
I_j^{3}\lesssim&\sum_{j'\geq j-2}2^{j'}\|\dot{\Delta}_{j'}a\|_{L^{p}}\|\dot{\Delta}_{j}\nabla\Pi\|_{L^{\infty}}\|\dot{\Delta}_j\Pi\|_{L^p}^{p-1}\nonumber\\
\lesssim&d_j2^{j(2-\frac{2}{p})}\|a\|_{\dot{B}_{p,1}^{\frac{2}{p}}}\|\nabla\Pi\|_{L^2}\|\dot{\Delta}_j\Pi\|_{L^p}^{p-1},
\end{align*}
which together with \eqref{Ij12} implies \eqref{Ijp12}.

Next, if $p\in[2,4)$, using integration by parts leads to
\begin{align}\label{ibp}
I_j=-(p-1)\int_{\R^2}[\dot{\Delta}_j,a]\nabla\Pi\cdot|\dot{\Delta}_j\Pi|^{p-2}\nabla\dot{\Delta}_j\Pi dx
\lesssim 2^j\|[\dot{\Delta}_j,a]\nabla\Pi\|_{L^p}\|\dot{\Delta}_j\Pi\|_{L^p}^{p-1}.
\end{align}
Applying again Bony's decomposition gives
\begin{align*}
[\dot{\Delta}_j,a]\nabla\Pi=[\dot{\Delta}_j,\dot{T}_{a}]\nabla\Pi+\dot{\Delta}_j\big(\dot{T}_{\nabla\Pi}a\big)
+\dot{\Delta}_j\big(\dot{R}(a,\nabla\Pi)\big)-\dot{T}'_{\dot{\Delta}_j\nabla\Pi}a.
\end{align*}
Thanks to \eqref{cmt}, we get by using Young's inequality and Lemma \ref{Bernstein Lemma} that
\begin{align*}
\|[\dot{\Delta}_j,\dot{T}_{a}]\nabla\Pi\|_{L^p}
\lesssim2^{-j}\sum_{|j'-j|\leq4}\|\nabla\dot{S}_{j'-1}a\|_{L^\infty}\|\dot{\Delta}_{j'}\nabla\Pi\|_{L^p}
\lesssim d_j2^{j(1-\frac{2}{p})}\|a\|_{\dot{B}_{p,1}^{\frac{2}{p}}}\|\nabla\Pi\|_{\dot{B}_{p,2}^{\frac{2}{p}-1}}.
\end{align*}
Using again Lemma \ref{Bernstein Lemma}, we arrive at
\begin{align*}
\|\dot{\Delta}_j\big(\dot{T}_{\nabla\Pi}a\big)\|_{L^p}\lesssim&
\sum_{|j'-j|\leq4}\|\dot{\Delta}_{j'}a\|_{L^{p}}\|\dot{S}_{j'-1}\nabla\Pi\|_{L^\infty}
\lesssim d_j2^{j(1-\frac{2}{p})}\|a\|_{\dot{B}_{p,1}^{\frac{2}{p}}}\|\nabla\Pi\|_{\dot{B}_{p,2}^{\frac{2}{p}-1}},\\
\|\dot{T}'_{\dot{\Delta}_j\nabla\Pi}a\|_{L^p}\lesssim&
\sum_{j'\geq j-2}\|\dot{\Delta}_{j'}a\|_{L^{p}}\|\dot{\Delta}_{j}\nabla\Pi\|_{L^\infty}
\lesssim d_j2^{j(1-\frac{2}{p})}\|a\|_{\dot{B}_{p,1}^{\frac{2}{p}}}\|\nabla\Pi\|_{\dot{B}_{p,2}^{\frac{2}{p}-1}},
\end{align*}
and for $p\in[2,4)$
\begin{align*}
\|\dot{\Delta}_j\big(\dot{R}(a,\nabla\Pi)\big)\|_{L^p}\lesssim
2^{\frac{2j}{p}}\sum_{j'\geq j-3}\|\dot{\Delta}_{j'}a\|_{L^{p}}\|\widetilde{\dot{\Delta}}_{j'}\nabla\Pi\|_{L^p}
\lesssim d_j2^{j(1-\frac{2}{p})}\|a\|_{\dot{B}_{p,1}^{\frac{2}{p}}}\|\nabla\Pi\|_{\dot{B}_{p,2}^{\frac{2}{p}-1}}.
\end{align*}
Thus we obtain
$$
\|[\dot{\Delta}_j,a]\nabla\Pi\|_{L^p}\lesssim d_j2^{j(1-\frac{2}{p})}\|a\|_{\dot{B}_{p,1}^{\frac{2}{p}}}\|\nabla\Pi\|_{\dot{B}_{p,2}^{\frac{2}{p}-1}},
$$
which along with \eqref{ibp} gives \eqref{Ijp24}.
\end{proof}

Motivated by \cite{AGZ ARMA 2012,AGZ JMPA 2013,Danchin AFSTM 2006,Danchin JDE 2010},
we shall use Lemma \ref{Lemma3.1} and a duality argument to prove the following crucial elliptic estimates in two space dimensions.
\begin{Proposition}\label{Proposition3.1}
Let $(p,q)\in(1,\infty)\times[1,\infty)$, $a\in\dot{B}_{q,1}^{\frac{2}{q}}(\R^2)$ with $1+a\geq\kappa>0$.
Let $F=(F^1,F^2)\in\dot{B}_{p,1}^{\frac{2}{p}-1}(\R^2)$ and $\nabla\Pi\in\dot{B}_{p,1}^{\frac{2}{p}-1}(\R^2)$ solve \eqref{elliptic equation}.
Then there hold

\rm{(\romannumeral1)} If $p\in(\frac{1+\sqrt{17}}{4},2)$ and $\frac1p-\frac1q\leq\frac12$,
or $p\in(2,\frac{5+\sqrt{17}}{2})$ and $\frac1p+\frac1q\geq\frac12$, then
\begin{align}\label{elliptic estimate1}
\|\nabla\Pi\|_{\dot{B}_{p,2}^{\frac{2}{p}-1}}
\lesssim\big(1+\|a\|_{\dot{B}_{q,1}^{\frac{2}{q}}}\big)\|\mathcal{Q}F\|_{\dot{B}_{p,2}^{\frac{2}{p}-1}}.
\end{align}

\rm{(\romannumeral2)} If $p\in(1,4)$ and $a\in\dot{B}_{p,1}^{\frac{2}{p}}(\R^2)$, then
\begin{align}\label{elliptic estimate2}
\|\nabla\Pi\|_{\dot{B}_{p,1}^{\frac{2}{p}-1}}
\lesssim\big(1+\|a\|_{\dot{B}_{p,1}^{\frac{2}{p}}}\big)^k\|\mathcal{Q}F\|_{\dot{B}_{p,1}^{\frac{2}{p}-1}},
\end{align}
where $k=1$ if $p\in(1,2]$, and $k=2$ if $p\in(2,4)$.
\end{Proposition}
\begin{proof} For simplicity, we just present \emph{a priori} estimate for smooth enough functions $a$, $F$ and $\nabla\Pi$.
Density arguments make the following computations rigorous.
Thanks to $1+a\geq\kappa>0$ and $\diverg F=\diverg\mathcal{Q}F$, we readily deduce from \eqref{elliptic equation} that
\begin{align}\label{elliptic estimateL2}
\kappa\|\nabla\Pi\|_{L^2}\leq\|\mathcal{Q}F\|_{L^2}.
\end{align}
Applying $\dot{\Delta}_j$ to \eqref{elliptic equation} gives
\begin{align*}
\diverg\big((1+a)\dot{\Delta}_j\nabla\Pi\big)=&\diverg\dot{\Delta}_jF-\diverg\big([\dot{\Delta}_j,a]\nabla\Pi\big).
\end{align*}
We next multiply the above equation by $-|\dot{\Delta}_j\Pi|^{p-2}\dot{\Delta}_j\Pi$ and integrate over $\R^2$.
Then applying Lemma 8 in Appendix B of \cite{Danchin JDE 2010} implies for some constants $c$ and $C$
\begin{align}\label{Ee1}
c\kappa2^{2j}\|\dot{\Delta}_j\Pi\|_{L^p}^p\leq C2^j\|\dot{\Delta}_j\mathcal{Q}F\|_{L^p}\|\dot{\Delta}_j\Pi\|_{L^p}^{p-1}
+\int_{\R^2}\diverg\big([\dot{\Delta}_j,a]\nabla\Pi\big)\cdot|\dot{\Delta}_j\Pi|^{p-2}\dot{\Delta}_j\Pi dx.
\end{align}

\rm{(\romannumeral1)} If $p\in(\frac{1+\sqrt{17}}{4},2)$ with $\frac1p-\frac1q\leq\frac12$, substituting \eqref{Ijpq} into \eqref{Ee1} leads to
\begin{align*}
\|\dot{\Delta}_j\nabla\Pi\|_{L^p}\lesssim\|\dot{\Delta}_j\mathcal{Q}F\|_{L^p}+d_j2^{j(1-\frac{2}{p})}\|a\|_{\dot{B}_{q,1}^{\frac{2}{q}}}\|\nabla\Pi\|_{L^2},
\end{align*}
which along with \eqref{elliptic estimateL2} and the embedding $\dot{B}_{p,2}^{\frac{2}{p}-1}(\R^2)\hookrightarrow L^2(\R^2)$, $l^1\hookrightarrow l^2$ gives
\begin{align}\label{Ee2}
\|\nabla\Pi\|_{\dot{B}_{p,2}^{\frac{2}{p}-1}}
\lesssim\big(1+\|a\|_{\dot{B}_{q,1}^{\frac{2}{q}}}\big)\|\mathcal{Q}F\|_{\dot{B}_{p,2}^{\frac{2}{p}-1}}.
\end{align}

Next, we consider the case when $p\in(2,\frac{5+\sqrt{17}}{2})$ with $\frac1p+\frac1q\geq\frac12$.
In this case, motivated by \cite{AGZ ARMA 2012,AGZ JMPA 2013,Danchin AFSTM 2006}, we shall use a duality argument:
\begin{align}\label{Ee3}
\|\nabla\Pi\|_{\dot{B}_{p,2}^{\frac{2}{p}-1}}=\sup_{\|g\|_{\dot{B}_{p',2}^{1-\frac{2}{p}}}=1}\langle\nabla\Pi,g\rangle
=\sup_{\|g\|_{\dot{B}_{p',2}^{\frac{2}{p'}-1}}=1}-\langle\Pi,\diverg g\rangle,
\end{align}
where $\langle\cdot,\cdot\rangle$ denotes the duality bracket between $\mathscr{S}'(\R^2)$ and $\mathscr{S}(\R^2)$.
Notice that $p'\in(\frac{1+\sqrt{17}}{4},2)$ and $\frac{1}{p'}-\frac1q\leq\frac12$,
then applying \eqref{Ee2} ensures that for any $g\in\dot{B}_{p',2}^{\frac{2}{p'}-1}(\R^2)$,
there exists a unique solution $\nabla P_g\in\dot{B}_{p',2}^{\frac{2}{p'}-1}(\R^2)$ to the elliptic equation
\begin{align*}
\diverg\big((1+a)\nabla P_g\big)=\diverg g,
\end{align*}
such that
\begin{align}\label{Ee4}
\|\nabla P_g\|_{\dot{B}_{p',2}^{\frac{2}{p'}-1}}
\lesssim\big(1+\|a\|_{\dot{B}_{q,1}^{\frac{2}{q}}}\big)\|g\|_{\dot{B}_{p',2}^{\frac{2}{p'}-1}}.
\end{align}
We proceed
\begin{align*}
-\langle\Pi,\diverg g\rangle=&-\langle\Pi,\diverg\big((1+a)\nabla P_g\big)\rangle=-\langle\diverg\big((1+a)\nabla\Pi\big),P_g\rangle\\
=&-\langle\diverg F,P_g\rangle=\langle\mathcal{Q}F,\nabla P_g\rangle
\leq\|\mathcal{Q}F\|_{\dot{B}_{p,2}^{\frac{2}{p}-1}}\|\nabla P_g\|_{\dot{B}_{p',2}^{\frac{2}{p'}-1}},
\end{align*}
which along with \eqref{Ee3} and \eqref{Ee4} implies \eqref{elliptic estimate1}.

\rm{(\romannumeral2)} If $p\in(1,2]$, substituting \eqref{Ijp12} into \eqref{Ee1} and using a similar argument as \eqref{Ee2} gives
\begin{align*}
\|\nabla\Pi\|_{\dot{B}_{p,1}^{\frac{2}{p}-1}}
\lesssim\big(1+\|a\|_{\dot{B}_{p,1}^{\frac{2}{p}}}\big)\|\mathcal{Q}F\|_{\dot{B}_{p,1}^{\frac{2}{p}-1}}.
\end{align*}
If $p\in(2,4)$, we get by substituting \eqref{Ijp24} into \eqref{Ee1} that
\begin{align*}
\|\dot{\Delta}_j\nabla\Pi\|_{L^p}\lesssim\|\dot{\Delta}_j\mathcal{Q}F\|_{L^p}+d_j2^{j(1-\frac{2}{p})}
\|a\|_{\dot{B}_{p,1}^{\frac{2}{p}}}\|\nabla\Pi\|_{\dot{B}_{p,2}^{\frac{2}{p}-1}},
\end{align*}
which along with \eqref{elliptic estimate1} leads to
\begin{align*}
\|\nabla\Pi\|_{\dot{B}_{p,1}^{\frac{2}{p}-1}}\lesssim
\big(1+\|a\|_{\dot{B}_{p,1}^{\frac{2}{p}}}\big)^2\|\mathcal{Q}F\|_{\dot{B}_{p,1}^{\frac{2}{p}-1}}.
\end{align*}
This completes the proof of the proposition.
\end{proof}

\begin{Proposition}\label{Proposition3.2}
Let $\mu>0$, $1<p<4$, $u_0\in\dot{B}_{p,1}^{\frac{2}{p}-1}(\R^2)$ and $a,b\in L_T^\infty(\dot{B}_{p,1}^{\frac{2}{p}}(\R^2))$
with $1+a\geq\kappa>0$ and $\mu+b\geq\kappa>0$.
Let $f\in L_T^1(\dot{B}_{p,1}^{\frac{2}{p}-1}(\R^2))$, $g\in L_T^1(\dot{B}_{p,1}^{\frac{2}{p}}(\R^2))$
and $\partial_t g=\diverg R$ with $R\in L_T^1(\dot{B}_{p,1}^{\frac{2}{p}-1}(\R^2))$.
Let $(u,\nabla\Pi)\in C([0,T];\dot{B}_{p,1}^{\frac{2}{p}-1}(\R^2))\cap L_T^1(\dot{B}_{p,1}^{\frac{2}{p}+1}(\R^2))
\times L_T^1(\dot{B}_{p,1}^{\frac{2}{p}-1}(\R^2))$ solve
\begin{eqnarray}\label{stokes}
\left\{\begin{aligned}
&\partial_t u-\mathrm{div}\big((\mu+b)\nabla u\big)+(1+a)\nabla\Pi=f,\\
&\diverg u=g,\\
&u|_{t=0}=u_0.
\end{aligned}\right.
\end{eqnarray}
Then there holds for $t\in[0,T]$
\begin{align}\label{parabolic estimate1}
&\|u\|_{\widetilde{L}_t^\infty(\dot{B}_{p,1}^{\frac{2}{p}-1})}+\|u\|_{L_t^1(\dot{B}_{p,1}^{\frac{2}{p}+1})}
+\|\nabla\Pi\|_{L_t^1(\dot{B}_{p,1}^{\frac{2}{p}-1})}\nonumber\\
\lesssim&\|u_0\|_{\dot{B}_{p,1}^{\frac{2}{p}-1}}+\big(1+\|a\|_{L_t^\infty(\dot{B}_{p,1}^{\frac{2}{p}})}\big)^3
\big\{\|f\|_{L_t^1(\dot{B}_{p,1}^{\frac{2}{p}-1})}+\|R\|_{L_t^1(\dot{B}_{p,1}^{\frac{2}{p}-1})}\nonumber\\
&+(1+\|b\|_{L_t^\infty(L^\infty)})\|g\|_{L_t^1(\dot{B}_{p,1}^{\frac{2}{p}})}
+2^{m}\|b\|_{L_t^\infty(\dot{B}_{p,1}^{\frac{2}{p}})}\|u\|_{L_t^1(\dot{B}_{p,1}^{\frac{2}{p}})}\big\},
\end{align}
provided that
\begin{align}\label{smallness}
\big(1+\|a\|_{L_T^\infty(\dot{B}_{p,1}^{\frac{2}{p}})}\big)^3\|b-\dot{S}_m b\|_{L_T^\infty(\dot{B}_{p,1}^{\frac{2}{p}})}\leq c_0
\end{align}
for some sufficiently small positive constant $c_0$ and some integer $m\in \mathbb{Z}$.
\end{Proposition}
\begin{proof}
Motivated by \cite{AGZ ARMA 2012,AGZ JMPA 2013,Danchin CPDE 2007}, we first use the decomposition $\mathrm{Id}=\dot{S}_m+(\mathrm{Id}-\dot{S}_m)$
to turn the $u$ equation of \eqref{stokes} into
\begin{align}\label{Pe1}
\partial_t u-\mathrm{div}\big((\mu+\dot{S}_m b)\nabla u\big)+(1+a)\nabla\Pi=f+\mathrm{div}\big((b-\dot{S}_m b)\nabla u\big).
\end{align}
Thanks to $\mu+b\geq\kappa>0$ and \eqref{smallness}, we infer
\begin{align}\label{Pe2}
\mu+\dot{S}_m b=\mu+b+(\dot{S}_m b-b)\geq\frac12\kappa.
\end{align}
Then taking $\diverg$ to \eqref{Pe1} and using $\diverg u=g$ and $\partial_t g=\diverg R$ leads to
\begin{align*}
\diverg\big((1+a)\nabla\Pi\big)=\diverg\big(f+\mathrm{div}\big((b-\dot{S}_m b)\nabla u\big)-R+\mu\nabla g
+\nabla \dot{S}_m b\cdot\nabla u+\dot{S}_m b\Delta u\big),
\end{align*}
from which, we get by applying \eqref{elliptic estimate2} that
\begin{align}\label{Pe3}
\|\nabla\Pi\|_{L_t^1(\dot{B}_{p,1}^{\frac{2}{p}-1})}\lesssim&\big(1+\|a\|_{L_t^\infty(\dot{B}_{p,1}^{\frac{2}{p}})}\big)^2
\big\{\|f\|_{L_t^1(\dot{B}_{p,1}^{\frac{2}{p}-1})}+\|(b-\dot{S}_m b)\nabla u\|_{L_t^1(\dot{B}_{p,1}^{\frac{2}{p}})}\nonumber\\
&+\|R\|_{L_t^1(\dot{B}_{p,1}^{\frac{2}{p}-1})}+\|g\|_{L_t^1(\dot{B}_{p,1}^{\frac{2}{p}})}
+\|\nabla\dot{S}_m b\cdot\nabla u\|_{L_t^1(\dot{B}_{p,1}^{\frac{2}{p}-1})}\nonumber\\
&+\|\mathcal{Q}(\dot{S}_m b\Delta u)\|_{L_t^1(\dot{B}_{p,1}^{\frac{2}{p}-1})}\big\}.
\end{align}
Using product laws in Besov spaces, we get
\begin{align}\label{Pe4}
&\|(b-\dot{S}_m b)\nabla u\|_{L_t^1(\dot{B}_{p,1}^{\frac{2}{p}})}+\|\nabla\dot{S}_mb\cdot\nabla u\|_{L_t^1(\dot{B}_{p,1}^{\frac2p-1})}\nonumber\\
\lesssim&\|b-\dot{S}_m b\|_{L_t^\infty(\dot{B}_{p,1}^{\frac{2}{p}})}\|u\|_{L_t^1(\dot{B}_{p,1}^{\frac{2}{p}+1})}
+2^m\|b\|_{L_t^\infty(\dot{B}_{p,1}^{\frac{2}{p}})}\|u\|_{L_t^1(\dot{B}_{p,1}^{\frac{2}{p}})}.
\end{align}
Yet notice that $\mathcal{Q}=-\nabla(-\Delta)^{-1}\diverg$ and $\diverg u=g$, we get by applying Bony's decomposition
$$
\mathcal{Q}(\dot{S}_mb\Delta u)=-\nabla(-\Delta)^{-1}\big(\dot{T}_{\nabla \dot{S}_mb}\Delta u+\dot{T}_{\dot{S}_mb}\Delta g\big)
+\mathcal{Q}\big(\dot{T}_{\Delta u}\dot{S}_mb\big)+\mathcal{Q}\big(\dot{R}(\dot{S}_mb,\Delta u)\big).
$$
Then it is easy to get
\begin{align*}
\|\dot{\Delta}_j\big(\dot{T}_{\nabla \dot{S}_mb}\Delta u\big)\|_{L^p}
\lesssim&\sum_{|j'-j|\leq4}\|\dot{S}_{j'-1}\nabla \dot{S}_mb\|_{L^\infty}\|\dot{\Delta}_{j'}\Delta u\|_{L^p}
\lesssim d_j2^{j(2-\frac2p)+m}\|b\|_{L^\infty}\|u\|_{\dot{B}_{p,1}^{\frac{2}{p}}},\\
\|\dot{\Delta}_j\big(\dot{T}_{\dot{S}_mb}\Delta g\big)\|_{L^p}
\lesssim&\sum_{|j'-j|\leq4}\|\dot{S}_{j'-1}\dot{S}_mb\|_{L^\infty}\|\dot{\Delta}_{j'}\Delta g\|_{L^p}
\lesssim d_j2^{j(2-\frac2p)}\|b\|_{L^\infty}\|g\|_{\dot{B}_{p,1}^{\frac{2}{p}}},\\
\|\dot{\Delta}_j\big(\dot{T}_{\Delta u}\dot{S}_mb\big)\|_{L^p}
\lesssim&\sum_{|j'-j|\leq4}\|\dot{\Delta}_{j'}\dot{S}_mb\|_{L^p}\|\dot{S}_{j'-1}\Delta u\|_{L^\infty}
\lesssim d_j2^{j(1-\frac2p)+m}\|b\|_{\dot{B}_{p,1}^{\frac{2}{p}}}\|u\|_{\dot{B}_{p,1}^{\frac{2}{p}}}.
\end{align*}
While for $p\in[2,4)$, applying Lemma \ref{Bernstein Lemma} yields
\begin{align*}
\|\dot{\Delta}_j\dot{R}(\dot{S}_mb,\Delta u)\|_{L^p}
\lesssim&2^{\frac{2j}{p}}\sum_{j'\geq j-3}\|\dot{\Delta}_{j'}\dot{S}_mb\|_{L^p}\|\widetilde{\dot{\Delta}}_{j'}\Delta u\|_{L^p}
\lesssim d_j2^{j(1-\frac2p)+m}\|b\|_{\dot{B}_{p,1}^{\frac{2}{p}}}\|u\|_{\dot{B}_{p,1}^{\frac{2}{p}}}.
\end{align*}
And for $p\in(1,2)$, we have $p'=\frac{p}{p-1}>p$ so that we could alternatively get
\begin{align*}
\|\dot{\Delta}_j\dot{R}(\dot{S}_mb,\Delta u)\|_{L^p}
\lesssim&2^{\frac{2j}{p'}}\sum_{j'\geq j-3}\|\dot{\Delta}_{j'}\dot{S}_mb\|_{L^{p'}}\|\widetilde{\dot{\Delta}}_{j'}\Delta u\|_{L^p}
\lesssim d_j2^{j(1-\frac2p)+m}\|b\|_{\dot{B}_{p,1}^{\frac{2}{p}}}\|u\|_{\dot{B}_{p,1}^{\frac{2}{p}}}.
\end{align*}
Whence we conclude that
\begin{align}\label{Pe5}
\|\mathcal{Q}(\dot{S}_mb\Delta u)\|_{L_t^1(\dot{B}_{p,1}^{\frac2p-1})}
\lesssim2^{m}\|b\|_{L_t^\infty(\dot{B}_{p,1}^{\frac{2}{p}})}\|u\|_{L_t^1(\dot{B}_{p,1}^{\frac{2}{p}})}
+\|b\|_{L_t^\infty(L^\infty)}\|g\|_{L_t^1(\dot{B}_{p,1}^{\frac{2}{p}})}.
\end{align}
Substituting \eqref{Pe4} and \eqref{Pe5} into \eqref{Pe3}, we arrive at
\begin{align}\label{Pe6}
\|\nabla\Pi\|_{L_t^1(\dot{B}_{p,1}^{\frac{2}{p}-1})}\lesssim&\big(1+\|a\|_{L_t^\infty(\dot{B}_{p,1}^{\frac{2}{p}})}\big)^2
\big\{\|f\|_{L_t^1(\dot{B}_{p,1}^{\frac{2}{p}-1})}+\|R\|_{L_t^1(\dot{B}_{p,1}^{\frac{2}{p}-1})}\nonumber\\
&+(1+\|b\|_{L_t^\infty(L^\infty)})\|g\|_{L_t^1(\dot{B}_{p,1}^{\frac{2}{p}})}
+2^{m}\|b\|_{L_t^\infty(\dot{B}_{p,1}^{\frac{2}{p}})}\|u\|_{L_t^1(\dot{B}_{p,1}^{\frac{2}{p}})}\nonumber\\
&+\|b-\dot{S}_m b\|_{L_t^\infty(\dot{B}_{p,1}^{\frac{2}{p}})}\|u\|_{L_t^1(\dot{B}_{p,1}^{\frac{2}{p}+1})}\big\}.
\end{align}
On the other hand, applying $\dot{\Delta}_j$ to \eqref{Pe1}, we arrive at
\begin{align*}
&\partial_t\dot{\Delta}_j u-\diverg\big((\mu+\dot{S}_m b)\dot{\Delta}_j\nabla u\big)+\dot{\Delta}_j\big((1+a)\nabla\Pi\big)\\
=&\dot{\Delta}_jf+\dot{\Delta}_j\mathrm{div}\big((b-\dot{S}_m b)\nabla u\big)+\diverg\big([\dot{\Delta}_j,\dot{S}_m b]\nabla u\big).
\end{align*}
Multiplying the $i$-th ($i=1,2$) equation by $|\dot{\Delta}_ju^i|^{p-2}\dot{\Delta}_ju^i$ and using a similar argument as \eqref{Ee1} leads to
\begin{align*}
\frac{d}{dt}\|\dot{\Delta}_j u\|_{L^p}+c\kappa2^{2j}\|\dot{\Delta}_j u\|_{L^p}
\lesssim&\|\dot{\Delta}_j f\|_{L^p}+2^j\|\dot{\Delta}_j\big((b-\dot{S}_m b)\nabla u\big)\|_{L^p}\nonumber\\
&+\|\diverg([\dot{\Delta}_j,\dot{S}_m b]\nabla u)\|_{L^p}+\|\dot{\Delta}_j\big((1+a)\nabla\Pi\big)\|_{L^p}.
\end{align*}
After time integration, multiplying $2^{(\frac{2}{p}-1)j}$ and summing up over $j$, we get
\begin{align}\label{Pe7}
&\|u\|_{\widetilde{L}_t^\infty(\dot{B}_{p,1}^{\frac{2}{p}-1})}+\|u\|_{L_t^1(\dot{B}_{p,1}^{\frac{2}{p}+1})}\nonumber\\
\lesssim&\|u_0\|_{\dot{B}_{p,1}^{\frac{2}{p}-1}}+\|f\|_{L_t^1(\dot{B}_{p,1}^{\frac{2}{p}-1})}
+\|b-\dot{S}_m b\|_{L_t^\infty(\dot{B}_{p,1}^{\frac{2}{p}})}\|u\|_{L_t^1(\dot{B}_{p,1}^{\frac{2}{p}+1})}\nonumber\\
&+2^{m}\|b\|_{L_t^\infty(\dot{B}_{p,1}^{\frac{2}{p}})}\|u\|_{L_t^1(\dot{B}_{p,1}^{\frac{2}{p}})}
+(1+\|a\|_{L_t^\infty(\dot{B}_{p,1}^{\frac{2}{p}})}\big)\|\nabla\Pi\|_{L_t^1(\dot{B}_{p,1}^{\frac{2}{p}-1})}.
\end{align}
where we used product laws and Lemma 5 in the appendix of \cite{Danchin AIF 2014}
\begin{align*}
\sum_{j\in\mathbb{Z}}2^{j(\frac{2}{p}-1)}\|\diverg([\dot{\Delta}_j,\dot{S}_m b]\nabla u)\|_{L_t^1(L^p)}
\lesssim2^{m}\|b\|_{L_t^\infty(\dot{B}_{p,1}^{\frac{2}{p}})}\|u\|_{L_t^1(\dot{B}_{p,1}^{\frac{2}{p}})}.
\end{align*}
Combining \eqref{Pe6} and \eqref{Pe7} and using \eqref{smallness}, we conclude the proof of \eqref{parabolic estimate1}.
\end{proof}

The following corollary will be used to prove the local well-posedness part of Theorem \ref{Theorem1.1}.
\begin{Corollary}\label{Corollary3.1}
Let $p,u_0,a,f,g,R$ be given in Proposition \ref{Proposition3.2} and $\|a\|_{L^\infty}\leq C_0$.
Let $(u,\nabla\Pi)\in C([0,T];\dot{B}_{p,1}^{\frac{2}{p}-1}(\R^2))\cap L_T^1(\dot{B}_{p,1}^{\frac{2}{p}+1}(\R^2))\times
L_T^1(\dot{B}_{p,1}^{\frac{2}{p}-1}(\R^2))$ solve
\begin{eqnarray*}
\left\{\begin{aligned}
&\partial_t u-(1+a)\mathrm{div}\big(2\tilde{\mu}(a)\mathcal{M}(u)\big)+(1+a)\nabla\Pi=f,\\
&\diverg u=g,\\
&u|_{t=0}=u_0,
\end{aligned}\right.
\end{eqnarray*}
with some smooth, positive function $\tilde{\mu}(a)$.
Further, denote by $b=b(a)\stackrel{\mathrm{def}}{=}(1+a)\tilde{\mu}(a)-\tilde{\mu}(0)$ and $\lambda=\lambda(a)\stackrel{\mathrm{def}}{=}\int_0^a\tilde{\mu}(s)ds$.
If there exist some sufficiently small positive constant $c_0$ and some integer $m\in \mathbb{Z}$ such that
\begin{align}\label{smallness1}
\big(1+\|a\|_{L_T^\infty(\dot{B}_{p,1}^{\frac{2}{p}})}\big)^3\|b-\dot{S}_m b,\lambda-\dot{S}_m\lambda\|_{L_T^\infty(\dot{B}_{p,1}^{\frac{2}{p}})}\leq c_0,
\end{align}
then we have for $t\in[0,T]$
\begin{align}\label{parabolic estimate2}
&\|u\|_{\widetilde{L}_t^\infty(\dot{B}_{p,1}^{\frac{2}{p}-1})}+\|u\|_{L_t^1(\dot{B}_{p,1}^{\frac{2}{p}+1})}
+\|\nabla\Pi\|_{L_t^1(\dot{B}_{p,1}^{\frac{2}{p}-1})}\nonumber\\
\lesssim&\|u_0\|_{\dot{B}_{p,1}^{\frac{2}{p}-1}}+\big(1+\|a\|_{L_t^\infty(\dot{B}_{p,1}^{\frac{2}{p}})}\big)^3
\big\{\|f\|_{L_t^1(\dot{B}_{p,1}^{\frac{2}{p}-1})}+\|R\|_{L_t^1(\dot{B}_{p,1}^{\frac{2}{p}-1})}\nonumber\\
&+\big(1+\|a\|_{L_t^\infty(\dot{B}_{p,1}^{\frac{2}{p}})}\big)\|g\|_{L_t^1(\dot{B}_{p,1}^{\frac{2}{p}})}
+2^{m}\|a\|_{L_t^\infty(\dot{B}_{p,1}^{\frac{2}{p}})}\|u\|_{L_t^1(\dot{B}_{p,1}^{\frac{2}{p}})}\big\}.
\end{align}
\end{Corollary}
\begin{proof} Since $b(a)$ and $\lambda(a)$ are smooth functions with $b(0)=\lambda(0)=0$, we get by applying Theorem 2.61 in \cite{BCD} that
\begin{align}\label{coro1}
\|b,\lambda\|_{\dot{B}_{p,1}^{\frac{2}{p}}}
\lesssim(1+\|a\|_{L^\infty})^{[\frac2p]+1}\|b',\lambda'\|_{W^{[\frac2p]+1,\infty}}\|a\|_{\dot{B}_{p,1}^{\frac{2}{p}}}
\lesssim\|a\|_{\dot{B}_{p,1}^{\frac{2}{p}}}.
\end{align}
Thanks to $\diverg u=g$, we rewrite the equation for $u$ as follow:
\begin{align}\label{coro2}
&\partial_t u-\mathrm{div}\big((\tilde{\mu}(0)+b)\nabla u\big)+(1+a)\nabla\Pi\nonumber\\
&=\tilde{f}\stackrel{\mathrm{def}}{=}f+(\tilde{\mu}(0)+b)\nabla g+\nabla u\cdot\nabla b-2\mathcal{M}(u)\cdot\nabla\lambda,
\end{align}
where $(\nabla u\cdot\nabla b)_i\stackrel{\mathrm{def}}{=}\partial_i u\cdot\nabla b=\partial_i u^j\partial_j b$.
While applying \eqref{coro1} and product laws in Besov spaces gives rise to
\begin{align*}
\|\tilde{f}\|_{L_t^1(\dot{B}_{p,1}^{\frac{2}{p}-1})}\lesssim&\|f\|_{L_t^1(\dot{B}_{p,1}^{\frac{2}{p}-1})}
+\big(1+\|a\|_{L_t^\infty(\dot{B}_{p,1}^{\frac{2}{p}})}\big)\|g\|_{L_t^1(\dot{B}_{p,1}^{\frac{2}{p}})}
+2^m\|a\|_{L_t^\infty(\dot{B}_{p,1}^{\frac{2}{p}})}\|u\|_{L_t^1(\dot{B}_{p,1}^{\frac{2}{p}})}\nonumber\\
&+\|b-\dot{S}_ma,\lambda-\dot{S}_m\lambda\|_{L_t^\infty(\dot{B}_{p,1}^{\frac{2}{p}})}\|u\|_{L_t^1(\dot{B}_{p,1}^{\frac{2}{p}+1})},
\end{align*}
from which and \eqref{smallness1}, we apply Proposition \ref{Proposition3.2} to \eqref{coro2} to conclude the proof of \eqref{parabolic estimate2}.
\end{proof}

\section{\bf Local well-posedness of \eqref{Model2}}\label{Section4}

\subsection{Local existence}
\noindent{\bf Step 1.} Construction of smooth approximate solutions.

Firstly, there exists a sequence $\{(a_0^n,\tilde{u}_0^n)\}_{n\in\mathbb{N}}\subset\mathscr{S}(\R^2)$ such that $(a_0^n,\tilde{u}_0^n)$
converges to $(a_0,u_0)$ in $\dot{B}_{p,1}^{\frac2p}(\R^2)\times\dot{B}_{p,1}^{\frac2p-1}(\R^2)$.
Define $u_0^n\stackrel{\mathrm{def}}{=}\mathbb{P}\tilde{u}_0^n$ so that $\diverg u_0^n=0$.
Then $u_0^n$ belongs to $H^\infty(\R^2)$ and converges to $u_0$ in $\dot{B}_{p,1}^{\frac2p-1}(\R^2)$.
Furthermore, we could assume that
\begin{align*}
\|a_0^n\|_{L^\infty}\leq2\|a_0\|_{L^\infty},\ \|a_0^n\|_{\dot{B}_{p,1}^{\frac2p}}\leq&2\|a_0\|_{\dot{B}_{p,1}^{\frac2p}},\
\|u_0^n\|_{\dot{B}_{p,1}^{\frac2p-1}}\leq2\|u_0\|_{\dot{B}_{p,1}^{\frac2p-1}},\\
\rm{and}\ \ 1+a_0^n=&1+a_0+(a_0^n-a_0)\geq\frac12\kappa.
\end{align*}
Whence applying Theorem 1.2 of \cite{Abidi RMI 2007} ensures that the system \eqref{Model2} with the initial data
$(a_0^n,u_0^n)$ admits a unique local in time solution $(a^n,u^n,\nabla\Pi^n)$ belonging to
\begin{align*}
C([0,T^n);H^{\alpha+1}(\R^2))&\times\big(C([0,T^n);H^{\alpha}(\R^2))\cap\widetilde{L}_{loc}^1(0,T^n;H^{\alpha+2}(\R^2))\big)\\
&\times\widetilde{L}_{loc}^1(0,T^n;H^{\alpha}(\R^2))\ \ \rm{with}\ \  \alpha>0.
\end{align*}
Moreover, we deduce from the transport equation of \eqref{Model2} that
\begin{align}\label{s1an}
1+\inf_{(t,x)\in[0,T^n)\times\R^2}a^n(t,x)=1+\inf_{y\in\R^2}a_0^n(y)\geq\frac12\kappa,\nonumber\\
\|a^n(t)\|_{L^\infty}=\|a_0^n\|_{L^\infty}\leq2\|a_0\|_{L^\infty},\ \ \forall t\in[0,T^n).
\end{align}

\noindent{\bf Step 2.} Uniform estimates to the approximate solutions.

Next, we shall prove that there exists a positive time $T<\inf_{n\in\mathbb{N}}T^n$
such that $(a^n,u^n,\nabla\Pi^n)$ is uniformly bounded in the space
$$
E_T\stackrel{\mathrm{def}}{=}\widetilde{L}_T^\infty(\dot{B}_{p,1}^{\frac2p})
\times\big(\widetilde{L}_T^\infty(\dot{B}_{p,1}^{\frac2p-1})\cap L_T^1(\dot{B}_{p,1}^{\frac2p+1})\big)
\times L_T^1(\dot{B}_{p,1}^{\frac2p-1}).
$$

For this, let $(u_L(t),u_L^n(t)\big)\stackrel{\mathrm{def}}{=}e^{\mu t\Delta}\big(u_0,u_0^n\big)$
with $\mu\stackrel{\mathrm{def}}{=}\tilde{\mu}(0)$. Then it is easy to observe that
\begin{align}\label{s2ufn1}
\|u_L^n\|_{\widetilde{L}^\infty(\R^+;\dot{B}_{p,1}^{\frac2p-1})}+\mu\|u_L^n\|_{L^1(\R^+;\dot{B}_{p,1}^{\frac2p+1})}
\lesssim\|u_0^n\|_{\dot{B}_{p,1}^{\frac2p-1}}\lesssim\|u_0\|_{\dot{B}_{p,1}^{\frac2p-1}},
\end{align}
and
\begin{align*}
\|u_L^n\|_{L_T^1(\dot{B}_{p,1}^{\frac2p+1})}\leq\|u_L\|_{L_T^1(\dot{B}_{p,1}^{\frac2p+1})}+C\|u_0^n-u_0\|_{\dot{B}_{p,1}^{\frac2p-1}}.
\end{align*}
Whence for any $\varepsilon>0$, there exist a number $k=k(\varepsilon)\in\mathbb{N}$
and a positive time $T=T(\varepsilon,u_0)$ such that
\begin{align}\label{s2ufn2}
\sup_{n\geq k}\|u_L^n\|_{L_{T}^1(\dot{B}_{p,1}^{\frac2p+1})}\leq\varepsilon.
\end{align}
Denote by $\bar{u}^n\stackrel{\mathrm{def}}{=}u^n-u_L^n$. Then the system for $(a^n,\bar{u}^n,\nabla\Pi^n)$ reads
\begin{eqnarray}\label{Modeln}
\left\{\begin{aligned}
&\partial_t a^n+(u_L^n+\bar{u}^n)\cdot\nabla a^n =0,\\
&\partial_t \bar{u}^n-(1+a^n)\diverg\big(2\tilde{\mu}(a^n)\mathcal{M}(\bar{u}^n)\big)+(1+a^n)\nabla\Pi^n=F_n,\\
&\diverg \bar{u}^n=0,\\
&(a^n,\bar{u}^n)|_{t=0}=(a_0^n,0),
\end{aligned}\right.
\end{eqnarray}
where
\begin{align*}
F_n=&-\bar{u}^n\cdot\nabla \bar{u}^n-u_L^n\cdot\nabla u_L^n-\diverg(\bar{u}^n\otimes u_L^n+u_F^n\otimes\bar{u}^n)\\
&+\diverg\big(2(\tilde{\mu}(a^n)-\tilde{\mu}(0))\mathcal{M}(u_L^n)\big)+a^n\diverg\big(2\tilde{\mu}(a^n)\mathcal{M}(u_L^n)\big).
\end{align*}

For notational simplicity, we denote by $A^n(t)\stackrel{\mathrm{def}}{=}\|a^n\|_{\widetilde{L}_t^\infty(\dot{B}_{p,1}^{\frac2p})}$ and
$$
Z^n(t)\stackrel{\mathrm{def}}{=}\|\bar{u}^n\|_{\widetilde{L}_t^\infty(\dot{B}_{p,1}^{\frac2p-1})}
+\|\bar{u}^n\|_{L_t^1(\dot{B}_{p,1}^{\frac2p+1})}+\|\nabla\Pi^n\|_{L_t^1(\dot{B}_{p,1}^{\frac2p-1})}.
$$
Then thanks to \eqref{s2ufn1}, we get by applying product laws and interpolation inequality in Besov spaces that
\begin{align*}
\|\bar{u}^n\cdot\nabla \bar{u}^n+u_L^n\cdot\nabla u_L^n\|_{L_t^1(\dot{B}_{p,1}^{\frac2p-1})}
\lesssim&\big(Z^n(t)\big)^2+\|u_0\|_{\dot{B}_{p,1}^{\frac2p-1}}\|u_L^n\|_{L_t^1(\dot{B}_{p,1}^{\frac2p+1})},\\
\|\diverg(\bar{u}^n\otimes u_L^n+u_L^n\otimes\bar{u}^n)\|_{L_t^1(\dot{B}_{p,1}^{\frac2p-1})}\lesssim&\|u_L^n\|_{L_t^2(\dot{B}_{p,1}^{\frac2p})}Z^n(t).
\end{align*}
Along the same line, one has
\begin{align*}
\|\diverg\big(2(\tilde{\mu}(a^n)-\tilde{\mu}(0))\mathcal{M}(u_L^n)\big)\|_{L_t^1(\dot{B}_{p,1}^{\frac2p-1})}
\lesssim&A^n(t)\|u_L^n\|_{L_t^1(\dot{B}_{p,1}^{\frac2p+1})},\\
\|a^n\diverg\big(2\tilde{\mu}(a^n)\mathcal{M}(u_L^n)\big)\|_{L_t^1(\dot{B}_{p,1}^{\frac2p-1})}
\lesssim&A^n(t)(1+A^n(t))\|u_L^n\|_{L_t^1(\dot{B}_{p,1}^{\frac2p+1})}.
\end{align*}
As a consequence, we obtain
\begin{align}\label{s2Fn}
\|F_n\|_{L_t^1(\dot{B}_{p,1}^{\frac2p-1})}\lesssim&\big(Z^n(t)+\|u_L^n\|_{L_t^2(\dot{B}_{p,1}^{\frac2p})}\big)Z^n(t)\nonumber\\
&+\big(\|u_0\|_{\dot{B}_{p,1}^{\frac2p-1}}+A^n(t)(1+A^n(t))\big)\|u_L^n\|_{L_t^1(\dot{B}_{p,1}^{\frac2p+1})}.
\end{align}
Denote by $b^n\stackrel{\mathrm{def}}{=}(1+a^n)\tilde{\mu}(a^n)-\tilde{\mu}(0)$, $\lambda^n\stackrel{\mathrm{def}}{=}\int_0^{a^n}\tilde{\mu}(s)ds$.
Applying Corollary \ref{Corollary3.1} to the $\bar{u}^n$ equation of \eqref{Modeln}, we get for $t\in[0,T]\subset[0,T^n)$ that
\begin{align}\label{parabolic estimaten}
Z^n(t)\lesssim\big(1+A^n(t)\big)^3\big(\|F_n\|_{L_t^1(\dot{B}_{p,1}^{\frac{2}{p}-1})}+2^mA^n(t)\|\bar{u}^n\|_{L_t^1(\dot{B}_{p,1}^{\frac{2}{p}})}\big),
\end{align}
provided that
\begin{align}\label{smallnessn}
\big(1+A^n(T)\big)^3\|b^n-\dot{S}_mb^n,\lambda^n-\dot{S}_m\lambda^n\|_{L_T^\infty(\dot{B}_{p,1}^{\frac{2}{p}})}\leq c_0
\end{align}
for some sufficiently small positive constant $c_0$ and some integer $m\in \mathbb{Z}$.
Substituting \eqref{s2Fn} into \eqref{parabolic estimaten} and using interpolation, we arrive at
\begin{align}\label{s2Zn}
Z^n(t)\leq&C\big(1+A^n(t)\big)^4\big\{\big(Z^n(t)+\|u_L^n\|_{L_t^2(\dot{B}_{p,1}^{\frac2p})}+2^mt^\frac12\big)Z^n(t)\nonumber\\
&+\big(\|u_0\|_{\dot{B}_{p,1}^{\frac2p-1}}+A^n(t)\big)\|u_L^n\|_{L_t^1(\dot{B}_{p,1}^{\frac2p+1})}\big\}.
\end{align}
On the other hand, applying \eqref{transport estimate1} to the transport equation of \eqref{Modeln}, we have for $t\in[0,T^n)$
\begin{align}\label{s2An}
A^n(t)\leq C\|a_0\|_{\dot{B}_{p,1}^{\frac2p}}\exp\left(C\big(Z^n(t)+\|u_0\|_{\dot{B}_{p,1}^{\frac2p-1}}\big)\right).
\end{align}
However, for any function $\chi\in\mathcal{D}(\R)$ vanishing at 0,
the composite function $\chi(a^n)$ with initial data $\chi(a_0^n)$ also solves the renormalized transport equation
$$\partial_t\chi(a^n)+(u_L^n+\bar{u}^n)\cdot\nabla \chi(a^n)=0.$$
Then applying \eqref{transport estimate2} to the above equation gives rise to
\begin{align*}
&\|\chi(a^n)-\dot{S}_m \chi(a^n)\|_{\widetilde{L}_t^\infty(\dot{B}_{p,1}^{\frac2p})}\nonumber\\
\leq&\sum_{j\geq m}2^{\frac{2j}{p}}\|\dot{\Delta}_j \chi(a_0^n)\|_{L^p}+\|\chi(a_0^n)\|_{\dot{B}_{p,1}^{\frac2p}}
\left(\exp\big(CZ^n(t)+C\|u_L^n\|_{L_t^1(\dot{B}_{p,1}^{\frac2p+1})}\big)-1\right)\nonumber\\
\leq&\sum_{j\geq m}2^{\frac{2j}{p}}\|\dot{\Delta}_j \chi(a_0)\|_{L^p}
+C\big(1+\|a_0\|_{\dot{B}_{p,1}^{\frac2p}}\big)\|a_0^n-a_0\|_{\dot{B}_{p,1}^{\frac2p}}\nonumber\\
&+C\|a_0\|_{\dot{B}_{p,1}^{\frac2p}}\left(\exp\big(CZ^n(t)+C\|u_L^n\|_{L_t^1(\dot{B}_{p,1}^{\frac2p+1})}\big)-1\right),
\end{align*}
where we used
\begin{align*}
\|\chi(a_0^n)-\chi(a_0)\|_{\dot{B}_{p,1}^{\frac2p}}
=&\left\|(a_0^n-a_0)\int_0^1\chi'(\tau a_0^n+(1-\tau)a_0)d\tau\right\|_{\dot{B}_{p,1}^{\frac2p}}\\
\leq&C\big(1+\|a_0\|_{\dot{B}_{p,1}^{\frac2p}}\big)\|a_0^n-a_0\|_{\dot{B}_{p,1}^{\frac2p}}.
\end{align*}
As a consequence, we obtain for $t\in[0,T^n)$
\begin{align}\label{s2bn}
&\|b^n-\dot{S}_mb^n,\lambda^n-\dot{S}_m\lambda^n\|_{\widetilde{L}_t^\infty(\dot{B}_{p,1}^{\frac{2}{p}})}\nonumber\\
\leq&\sum_{j\geq m}2^{\frac{2j}{p}}\|\dot{\Delta}_jb_0,\dot{\Delta}_j\lambda_0\|_{L^p}
+C\big(1+\|a_0\|_{\dot{B}_{p,1}^{\frac2p}}\big)\|a_0^n-a_0\|_{\dot{B}_{p,1}^{\frac2p}}\nonumber\\
&+C\|a_0\|_{\dot{B}_{p,1}^{\frac2p}}\left(\exp\big(CZ^n(t)+C\|u_L^n\|_{L_t^1(\dot{B}_{p,1}^{\frac2p+1})}\big)-1\right),
\end{align}
with $b_0\stackrel{\mathrm{def}}{=}(1+a_0)\tilde{\mu}(a_0)-\tilde{\mu}(0)$, $\lambda_0\stackrel{\mathrm{def}}{=}\int_0^{a_0}\tilde{\mu}(s)ds$.

Next, for any $n\in\mathbb{N}$, we define
\begin{align}\label{s2Tnstar}
T_*^n\stackrel{\mathrm{def}}{=}\sup\{t\in(0,T^n):\ Z^n(t)\leq 2\varepsilon_0\}
\end{align}
with $\varepsilon_0\in(0,\frac12)$ to be determined. We shall prove $\inf_{n\in\mathbb{N}}T_*^n>0$.

Firstly, we deduce from \eqref{s2An} and \eqref{s2Tnstar} for $t\leq T_*^n$ that
\begin{align}\label{s2An1}
A^n(t)\leq C\|a_0\|_{\dot{B}_{p,1}^{\frac2p}}\exp\left(C\big(1+\|u_0\|_{\dot{B}_{p,1}^{\frac2p-1}}\big)\right)
\stackrel{\mathrm{def}}{=}A_0.
\end{align}
Notice that $(b_0,\lambda_0)\in\big(\dot{B}_{p,1}^{\frac2p}(\R^2)\big)^2$, there exist $m=m(c_0)\in\mathbb{Z}$ and $n_0=n_0(c_0)\in\mathbb{N}$ such that
\begin{align}
(1+A_0)^3\left(\sum_{j\geq m}2^{\frac{2j}{p}}\|\dot{\Delta}_jb_0,\dot{\Delta}_j\lambda_0\|_{L^p}
+C\big(1+\|a_0\|_{\dot{B}_{p,1}^{\frac2p}}\big)\sup_{n\geq n_0}\|a_0^n-a_0\|_{\dot{B}_{p,1}^{\frac2p}}\right)\leq\frac12 c_0.
\end{align}
Yet thanks to \eqref{s2ufn2}, taking $\varepsilon_0$ and $T_0$ small enough and $n_1\geq n_0$ large enough ensures
\begin{align}\label{s2smallness}
C(1+A_0)^3\|a_0\|_{\dot{B}_{p,1}^{\frac2p}}\left(\exp\big(2C\varepsilon_0
+C\sup_{n\geq n_1}\|u_L^n\|_{L_{T_0}^1(\dot{B}_{p,1}^{\frac2p+1})}\big)-1\right)\leq\frac12c_0.
\end{align}
Combining \eqref{s2An1}--\eqref{s2smallness} implies that \eqref{smallnessn} with $T=\min\{T_*^n,T_0\}$ is fulfilled for any $n\geq n_1$.
Without loss of generality, we may assume that $T_*^n\leq T_0$. Then for any $t\leq T_*^n$, we deduce from \eqref{s2Zn} that
\begin{align}\label{s2Zn1}
Z^n(t)\leq&C\big(1+A_0\big)^4
\big\{\big(2\varepsilon_0+\|u_L^n\|_{L_t^2(\dot{B}_{p,1}^{\frac2p})}+2^mt^\frac12\big)Z^n(t)\nonumber\\
&+\big(\|u_0\|_{\dot{B}_{p,1}^{\frac2p-1}}+A_0\big)\|u_L^n\|_{L_t^1(\dot{B}_{p,1}^{\frac2p+1})}\big\}.
\end{align}
Finally, taking $\varepsilon_0$ and $T_1$ small enough and $n_2\geq n_1$ large enough ensures for any $n\geq n_2$
\begin{align*}
C(1+A_0\big)^4\big(2\varepsilon_0+\|u_L^n\|_{L_{T_1}^2(\dot{B}_{p,1}^{\frac2p})}+2^mT_1^\frac12\big)\leq\frac12,
\end{align*}
and
\begin{align*}
2C\big(1+A_0\big)^4\big(\|u_0\|_{\dot{B}_{p,1}^{\frac2p-1}}+A_0\big)\|u_L^n\|_{L_{T_1}^1(\dot{B}_{p,1}^{\frac2p+1})}\leq\varepsilon_0,
\end{align*}
which together with \eqref{s2Zn1} implies
$$
Z^n(t)\leq\varepsilon_0,\ \ \forall t\leq\min(T_*^n,T_1),\ n\geq n_2.
$$
However, by the definition of $T_*^n$, we eventually conclude $T_*^n\geq T_1$ and $\sup_{n\geq n_2}Z^n(T_1)\leq \varepsilon_0$,
which along with \eqref{s2ufn1} and \eqref{s2An1} ensures that
\begin{align}\label{s2ubdd}
(a^n,u^n,\nabla\Pi^n)\ \ {\rm is\ uniformly\ bounded\ in}\ E_{T_1}.
\end{align}

\noindent{\bf Step 3.} Convergence.

Thanks to \eqref{s2ubdd}, we can repeat the compactness argument in Step 3 to the proof of Theorem 5.1 in \cite{Danchin PRSES 2003} to conclude that
$(a^n,u^n,\nabla\Pi^n)$ converges to some limit $(a,u,\nabla\Pi)$ which satisfies \eqref{soln regularity} and solves the system \eqref{Model2}
on $[0,T_1]$. Moreover, there exist some sufficiently small positive constant $c_0$ and some integer $m\in \mathbb{Z}$ such that
\begin{align}\label{s3smallness}
\|a-\dot{S}_m a\|_{L_{T_1}^\infty(\dot{B}_{p,1}^{\frac{2}{p}})}\leq c_0.
\end{align}

\subsection{Uniqueness}

As in \cite{DM CPAM 2012,DM ARMA 2013,Danchin AIF 2014}, we apply Lagrangian approach to prove the uniqueness part of Theorem \ref{Theorem1.1}.

Before going further, we give some notations. For a matrix $A=(A^{ij}):\R^2\rightarrow\R^{2\times2}$, we denote $A^T$ its transpose matrix,
$\mathrm{Tr}(A)$ its trace, $\det(A)$ its determinant and $(\diverg A)^j\stackrel{\mathrm{def}}{=}\partial_1A^{1j}+\partial_2A^{2j}$.
For a $C^1$ vector field $u:\R^2\rightarrow\R^2$, denote $(Du)^{ij}\stackrel{\mathrm{def}}{=}\partial_ju^i$,
$\nabla u\stackrel{\mathrm{def}}{=}(Du)^T$ and $\mathcal{M}_{A}(u)\stackrel{\mathrm{def}}{=}\frac12\big(Du\cdot A+A^T\cdot\nabla u\big)$.

Let $(a,u,\nabla\Pi)$ be a solution to \eqref{Model2} on $[0,T]$ and satisfy \eqref{soln regularity}.
By virtue of Cauchy-Lipschitz theorem, the unique trajectory $X(t,y)$ of $u$ is determined by
\begin{align*}
X(t,y)=y+\int_0^t u(\tau,X(\tau,y))d\tau,\ \ t\in[0,T]
\end{align*}
such that $X(t,\cdot)$ is a $C^1$-diffeomorphism over $\R^2$. Denote $A(t,y)\stackrel{\mathrm{def}}{=}\big(D_y X(t,y)\big)^{-1}$.
Then the divergence free condition for $u$ is equivalent to $\det(A)\equiv1$. To obtain the Lagrangian formulation of \eqref{Model2}, we define
\begin{align}\label{Lag coor}
(\eta,v,P)(t,y)\stackrel{\mathrm{def}}{=}(a,u,\Pi)\big(t,X(t,y)\big).
\end{align}
If $\|\nabla u\|_{L_T^1(\dot{B}_{p,1}^{\frac2p})}$ is sufficiently small,
then applying Proposition 8 in the appendix of \cite{Danchin AIF 2014} implies that $(\eta,v,\nabla P)$
belongs to the same functional space as $(a,u,\nabla\Pi)$. On the other hand, using the chain rule, we easily deduce that
\begin{align*}
\big(\partial_t a+u\cdot\nabla_x a\big)\big(t,X(t,y)\big)=&\partial_t \eta(t,y),\\
\big(\partial_t u+u\cdot\nabla_x u\big)\big(t,X(t,y)\big)=&\partial_t v(t,y),\\
\nabla_x\Pi\big(t,X(t,y)\big)=&A^{T}\nabla_y P(t,y).
\end{align*}
While applying Lemma A.1 in the appendix of \cite{DM CPAM 2012} gives
\begin{align}\label{divu}
\mathrm{div}_x u\big(t,X(t,y)\big)=\mathrm{Tr}(D_y v\cdot A)(t,y)=\mathrm{div}_y(Av)(t,y),
\end{align}
and
$$
\mathrm{div}_x\big(\tilde{\mu}(a)\mathcal{M}(u)\big)\big(t,X(t,y)\big)=\mathrm{div}_y\big(\tilde{\mu}(a_0)A\mathcal{M}_A(v)\big)(t,y).
$$
Whence we deduce from \eqref{Model2} that $\eta(t,\cdot)\equiv a_0$ and $(v,\nabla P)$ solves
\begin{eqnarray*}
\left\{\begin{aligned}
&\partial_t v-(1+a_0)\mathrm{div}\big(2\tilde{\mu}(a_0)A\mathcal{M}_{A}(v)\big)+(1+a_0)A^T\nabla P=0,\\
&\diverg(A v)=0,\\
&v|_{t=0}=u_0.
\end{aligned}\right.
\end{eqnarray*}

Now let $(a_i,u_i,\nabla\Pi_i)$ ($i=1,2$) be two solutions to \eqref{Model2} and $(\eta_i,v_i,P_i)$ be determined by \eqref{Lag coor}.
Denote $(\delta v,\delta P,\delta A)\stackrel{\mathrm{def}}{=}(v_2-v_1,P_2-P_1,A_2-A_1)$, where
\begin{align*}
A_i(t,y)\stackrel{\mathrm{def}}{=}\left(\mathrm{Id}+\int_0^tDv_i(\tau,y)d\tau\right)^{-1},\ \ \mathrm{for}\ i=1,2.
\end{align*}
Then the system for $(\delta v,\nabla\delta P)$ reads
\begin{eqnarray}\label{Modeldv}
\left\{\begin{aligned}
&\partial_t\delta v-(1+a_0)\mathrm{div}\big(2\tilde{\mu}(a_0)\mathcal{M}(\delta v)\big)+(1+a_0)\nabla\delta P=(1+a_0)\delta F,\\
&\diverg\delta v=g,\\
&\delta v|_{t=0}=0,
\end{aligned}\right.
\end{eqnarray}
where $g\stackrel{\mathrm{def}}{=}\diverg\big((\mathrm{Id}-A_2)\delta v-\delta A v_1\big)$,
$R\stackrel{\mathrm{def}}{=}-\partial_tA_2\delta v+(\mathrm{Id}-A_2)\partial_t\delta v-\partial_t\delta A v_1-\delta A \partial_tv_1$,
and $\delta F\stackrel{\mathrm{def}}{=}\sum_{k=1}^{6}\delta F_k$ with
\begin{align*}
\delta F_1\stackrel{\mathrm{def}}{=}&(\mathrm{Id}-A_2^T)\nabla\delta P, &\delta F_2\stackrel{\mathrm{def}}{=}&-\delta A^T\nabla P_1,\\
\delta F_3\stackrel{\mathrm{def}}{=}&\mathrm{div}\big(2\tilde{\mu}(a_0)(A_2-\mathrm{Id})\mathcal{M}(\delta v)\big),
&\delta F_4\stackrel{\mathrm{def}}{=}&\mathrm{div}\big(2\tilde{\mu}(a_0)A_2\mathcal{M}_{A_2-\mathrm{Id}}(\delta v)\big),\\
\delta F_5\stackrel{\mathrm{def}}{=}&\mathrm{div}\big(2\tilde{\mu}(a_0)A_2\mathcal{M}_{\delta A}(v_1)\big),
&\delta F_6\stackrel{\mathrm{def}}{=}&\mathrm{div}\big(2\tilde{\mu}(a_0)\delta A\mathcal{M}_{A_1}(v_1)\big).
\end{align*}

In the sequel, we shall take $T$ to be so small that
\begin{align*}
\int_0^T\|Dv_i(t)\|_{\dot{B}_{p,1}^{\frac{2}{p}}}dt\leq c,\ \ i=1,2,
\end{align*}
for some small enough constant $c$. Then the definition of $A_i$ implies that
\begin{align*}
A_i(t,y)=\mathrm{Id}+\sum_{k=1}^{\infty}(-1)^k\left(\int_0^tDv_i(\tau,y)d\tau\right)^k,\ \ i=1,2.
\end{align*}
Moreover, as proved in the appendix of \cite{DM CPAM 2012}, we have the following estimates:
\begin{align}
\|A_i-\mathrm{Id}\|_{L_t^\infty(\dot{B}_{p,1}^{\frac{2}{p}})}\lesssim&\|Dv_i\|_{L_t^1(\dot{B}_{p,1}^{\frac{2}{p}})},\ \ i=1,2,\label{u1}\\
\|\delta A\|_{L_t^\infty(\dot{B}_{p,1}^{\frac{2}{p}})}\lesssim&\|D\delta v\|_{L_t^1(\dot{B}_{p,1}^{\frac{2}{p}})},\label{u2}\\
\|\partial_tA_i\|_{\dot{B}_{p,1}^{\frac2p}}\lesssim&\|Dv_i\|_{\dot{B}_{p,1}^{\frac2p}},\ \ i=1,2,\label{u3}\\
\|\partial_t\delta A\|_{L_t^2(\dot{B}_{p,1}^{\frac2p-1})}
\lesssim&\|v_1,v_2\|_{L_t^2(\dot{B}_{p,1}^{\frac2p})}\|D\delta v\|_{L_t^1(\dot{B}_{p,1}^{\frac2p})}
+\|\delta v\|_{L_t^2(\dot{B}_{p,1}^{\frac2p})}.\label{u4}
\end{align}

We now choose $m$ to be so large that
\begin{align*}
\big(1+\|a_0\|_{\dot{B}_{p,1}^{\frac{2}{p}}}\big)^3\|b_0-\dot{S}_mb_0,\lambda_0-\dot{S}_m\lambda_0\|_{\dot{B}_{p,1}^{\frac{2}{p}}}\leq c_0,
\end{align*}
where $b_0\stackrel{\mathrm{def}}{=}(1+a_0)\tilde{\mu}(a_0)-\tilde{\mu}(0)$, $\lambda_0\stackrel{\mathrm{def}}{=}\int_0^{a_0}\tilde{\mu}(s)ds$.
Then \eqref{smallness1} is fulfilled so that we could apply Corollary \ref{Corollary3.1} to \eqref{Modeldv} to obtain
\begin{align}\label{udv}
&\|\delta v\|_{\widetilde{L}_t^\infty(\dot{B}_{p,1}^{\frac{2}{p}-1})}+\|\delta v\|_{L_t^1(\dot{B}_{p,1}^{\frac{2}{p}+1})}
+\|\nabla\delta P\|_{L_t^1(\dot{B}_{p,1}^{\frac{2}{p}-1})}\nonumber\\
\lesssim&\|\delta F\|_{L_t^1(\dot{B}_{p,1}^{\frac{2}{p}-1})}+\|R\|_{L_t^1(\dot{B}_{p,1}^{\frac{2}{p}-1})}
+\|g\|_{L_t^1(\dot{B}_{p,1}^{\frac{2}{p}})}+2^m\|\delta v\|_{L_t^1(\dot{B}_{p,1}^{\frac{2}{p}})}.
\end{align}
On the other hand, it is easy to deduce from \eqref{Modeldv} that
\begin{align}\label{uptdv}
\|\partial_t\delta v\|_{L_t^1(\dot{B}_{p,1}^{\frac{2}{p}-1})}\lesssim\|\delta v\|_{L_t^1(\dot{B}_{p,1}^{\frac{2}{p}+1})}
+\|\nabla\delta P\|_{L_t^1(\dot{B}_{p,1}^{\frac{2}{p}-1})}+\|\delta F\|_{L_t^1(\dot{B}_{p,1}^{\frac{2}{p}-1})}.
\end{align}
We denote $\delta E(t)\stackrel{\mathrm{def}}{=}\|\delta v\|_{\widetilde{L}_t^\infty(\dot{B}_{p,1}^{\frac{2}{p}-1})}
+\|\Delta\delta v,\nabla\delta P,\partial_t\delta v\|_{L_t^1(\dot{B}_{p,1}^{\frac{2}{p}-1})}.$
Summing up \eqref{udv} and \eqref{uptdv} and using interpolation in Besov spaces leads to
\begin{align}\label{udE}
\delta E(t)\lesssim\|\delta F\|_{L_t^1(\dot{B}_{p,1}^{\frac{2}{p}-1})}+\|R\|_{L_t^1(\dot{B}_{p,1}^{\frac{2}{p}-1})}
+\|g\|_{L_t^1(\dot{B}_{p,1}^{\frac{2}{p}})}+2^mt^{\frac12}\delta E(t).
\end{align}
We shall prove that $\delta E(t)=0$ for small enough $t$.

Now applying \eqref{u1}, \eqref{u2} and product laws in Besov spaces, we arrive at
\begin{align*}
\|\delta F_1\|_{L_t^1(\dot{B}_{p,1}^{\frac2p-1})}\lesssim&\|v_2\|_{L_t^1(\dot{B}_{p,1}^{\frac2p+1})}\delta E(t),\nonumber\\
\|\delta F_2\|_{L_t^1(\dot{B}_{p,1}^{\frac2p-1})}\lesssim&\|\nabla P_1\|_{L_t^1(\dot{B}_{p,1}^{\frac2p-1})}\delta E(t).
\end{align*}
Along the same line, one has
\begin{align*}
\|\delta F_3,\delta F_4,\delta F_5,\delta F_6\|_{L_t^1(\dot{B}_{p,1}^{\frac2p-1})}\lesssim\|v_1,v_2\|_{L_t^1(\dot{B}_{p,1}^{\frac2p+1})}\delta E(t).
\end{align*}
Thus, we obtain
\begin{align}\label{udF}
\|\delta F\|_{L_t^1(\dot{B}_{p,1}^{\frac2p-1})}\lesssim\|\Delta v_1,\Delta v_2,\nabla P_1\|_{L_t^1(\dot{B}_{p,1}^{\frac2p-1})}\delta E(t).
\end{align}
Using the same argument, we deduce from \eqref{u1}--\eqref{u4} that
\begin{align}
\|R\|_{L_t^1(\dot{B}_{p,1}^{\frac2p-1})}\lesssim
\big(\|\Delta v_2,\partial_tv_1\|_{L_t^1(\dot{B}_{p,1}^{\frac2p-1})}+\|v_1\|_{L_t^2(\dot{B}_{p,1}^{\frac2p})}\big)\delta E(t).
\end{align}
In order to bound $\delta g$ in $L_t^1(\dot{B}_{p,1}^{\frac2p})$, we apply \eqref{divu} to rewrite $\delta g$ as follow:
$$g=\mathrm{Tr}\big(D\delta v(\mathrm{Id}-A_2)-Dv_1\delta A\big),$$
from which, we easily get that
\begin{align}\label{udg}
\|g\|_{L_t^1(\dot{B}_{p,1}^{\frac2p})}\lesssim\|v_1,v_2\|_{L_t^1(\dot{B}_{p,1}^{\frac2p+1})}\delta E(t).
\end{align}
Plugging \eqref{udF}--\eqref{udg} into \eqref{udE}, we eventually get
\begin{align*}
\delta E(t)\lesssim&\big\{\|\Delta v_1,\Delta v_2,\nabla P_1,\partial_tv_1\|_{L_t^1(\dot{B}_{p,1}^{\frac2p-1})}
+\|v_1\|_{L_t^2(\dot{B}_{p,1}^{\frac2p})}+2^mt^{\frac12}\big\}\delta E(t),
\end{align*}
from which, we get by taking $t$ small enough that $\delta E(t)=0$. The uniqueness on $[0,T]$ can be obtained by a standard argument.

\section{\bf Global well-posedness of \eqref{Model2} with homogeneous viscosity}
In this section, we prove the global well-posedness part of Theorem \ref{Theorem1.1}.
\subsection{Higher regularities of the solutions}
\begin{Proposition}
Let $(a,u,\nabla\Pi)$ be the unique solution of \eqref{Modela} which satisfies \eqref{soln regularity}, \eqref{s3smallness}
and $1+a\geq\kappa>0$ . Then for any $t_0\in(0,T]$, there holds
\begin{align}\label{hr}
&\|u\|_{\widetilde{L}^\infty([t_0,t];\dot{B}_{p,1}^{\frac{2}{p}})}+\|u\|_{L^1([t_0,t];\dot{B}_{p,1}^{\frac{2}{p}+2})}
+\|\nabla\Pi\|_{L^1([t_0,t];\dot{B}_{p,1}^{\frac{2}{p}})}\nonumber\\
\leq&\big(t_0^{-\frac12}+2^{m}\|a\|_{L_t^\infty(\dot{B}_{p,1}^{\frac{2}{p}})}\big)e^{CZ(t)}
\end{align}
with $Z(t)\stackrel{\rm{def}}{=}\|u\|_{L_t^\infty(\dot{B}_{p,1}^{\frac{2}{p}-1})}+\|u\|_{L_t^1(\dot{B}_{p,1}^{\frac{2}{p}+1})}
+\|\nabla\Pi\|_{L_t^1(\dot{B}_{p,1}^{\frac{2}{p}-1})}$.
\end{Proposition}
\begin{proof} The proof of this proposition is similar to the proof of Proposition \ref{Proposition3.2}. For completeness, we outline its proof here.
We first rewrite the momentum equation of \eqref{Modela} as
\begin{align}\label{hr1}
\partial_t u-(1+\dot{S}_ma)\Delta u+(1+\dot{S}_ma)\nabla\Pi=-u\cdot\nabla u+\dot{E}_m
\end{align}
with $\dot{E}_m\stackrel{\rm{def}}{=}(a-\dot{S}_ma)(\Delta u-\nabla\Pi)$.
Thanks to $\diverg u=0$, applying $\dot{\Delta}_j\mathcal{P}$ to \eqref{hr1} leads to
\begin{align*}
&\partial_t\dot{\Delta}_j u-\diverg\big((1+\dot{S}_ma)\dot{\Delta}_j\nabla u\big)\\
=&\dot{\Delta}_j\mathcal{P}\big(-u\cdot\nabla u+\dot{E}_m+\Pi\nabla \dot{S}_ma\big)-\dot{\Delta}_j\big(\nabla\dot{S}_m a\cdot\nabla u\big)\\
&-\dot{\Delta}_j\mathcal{Q}\big(\dot{S}_ma\Delta u\big)+\diverg([\dot{\Delta}_j,\dot{S}_ma]\nabla u),
\end{align*}
from which, we infer for $0<\tau<t_0\leq t\leq T$ that
\begin{align}\label{hru}
&\|u\|_{\widetilde{L}^\infty([\tau,t];\dot{B}_{p,1}^{\frac{2}{p}})}+\|u\|_{L^1([\tau,t];\dot{B}_{p,1}^{\frac{2}{p}+2})}\nonumber\\
\lesssim&\|u(\tau)\|_{\dot{B}_{p,1}^{\frac{2}{p}}}+\|u\cdot\nabla u\|_{L^1([\tau,t];\dot{B}_{p,1}^{\frac{2}{p}})}
+\|\dot{E}_m\|_{L^1([\tau,t];\dot{B}_{p,1}^{\frac{2}{p}})}+\|\Pi\nabla \dot{S}_ma\|_{L^1([\tau,t];\dot{B}_{p,1}^{\frac{2}{p}})}\nonumber\\
&+\|\nabla \dot{S}_m a\cdot\nabla u\|_{L^1([\tau,t];\dot{B}_{p,1}^{\frac{2}{p}})}
+\|\nabla\dot{S}_ma\cdot\Delta u\|_{L^1([\tau,t];\dot{B}_{p,1}^{\frac{2}{p}-1})}\nonumber\\
&+\sum_{j\in\mathbb{Z}}2^{\frac{2j}{p}}\|\diverg([\dot{\Delta}_j,\dot{S}_ma]\nabla u)\|_{L^1([\tau,t];L^p)},
\end{align}
where we used $\mathcal{Q}(\dot{S}_m a\Delta u)=-\nabla(-\Delta)^{-1}(\nabla\dot{S}_ma\cdot\Delta u)$.
While applying $\dot{\Delta}_j\diverg$ to \eqref{hr1} gives
\begin{align*}
\diverg\big((1+\dot{S}_ma)\dot{\Delta}_j\nabla\Pi\big)=\dot{\Delta}_j\diverg\big(-u\cdot\nabla u+\dot{E}_m+\dot{S}_ma\Delta u\big)
-\diverg\big([\dot{\Delta}_j,\dot{S}_m a]\nabla\Pi\big),
\end{align*}
which implies
\begin{align}\label{hrPi}
\|\nabla\Pi\|_{L^1([\tau,t];\dot{B}_{p,1}^{\frac{2}{p}})}\lesssim&
\|u\cdot\nabla u\|_{L^1([\tau,t];\dot{B}_{p,1}^{\frac{2}{p}})}+\|\dot{E}_m\|_{L^1([\tau,t];\dot{B}_{p,1}^{\frac{2}{p}})}
+\|\nabla\dot{S}_ma\cdot\Delta u\|_{L^1([\tau,t];\dot{B}_{p,1}^{\frac{2}{p}-1})}\nonumber\\
&+\sum_{j\in\mathbb{Z}}2^{(\frac{2}{p}-1)j}\|\diverg\big([\dot{\Delta}_j,\dot{S}_m a]\nabla\Pi\big)\|_{L^1([\tau,t];L^p)}.
\end{align}
Summing up \eqref{hru} and \eqref{hrPi} and then applying product laws,
commutator estimates (see \cite[Lemma 5]{Danchin AIF 2014}) and \eqref{s3smallness}, we arrive at
\begin{align*}
&\|u\|_{\widetilde{L}^\infty([\tau,t];\dot{B}_{p,1}^{\frac{2}{p}})}+\|u\|_{L^1([\tau,t];\dot{B}_{p,1}^{\frac{2}{p}+2})}
+\|\nabla\Pi\|_{L^1([\tau,t];\dot{B}_{p,1}^{\frac{2}{p}})}\nonumber\\
\lesssim&\|u(\tau)\|_{\dot{B}_{p,1}^{\frac{2}{p}}}+2^{m}\|a\|_{L_t^\infty(\dot{B}_{p,1}^{\frac{2}{p}})}
\|\Delta u,\nabla\Pi\|_{L_t^1(\dot{B}_{p,1}^{\frac{2}{p}-1})}+\int_\tau^t\|u\|_{\dot{B}_{p,1}^{\frac{2}{p}}}\|u\|_{\dot{B}_{p,1}^{\frac{2}{p}+1}}dt',
\end{align*}
from which, we get by using Gronwall's inequality that
\begin{align*}
&\|u\|_{\widetilde{L}^\infty([\tau,t];\dot{B}_{p,1}^{\frac{2}{p}})}+\|u\|_{L^1([\tau,t];\dot{B}_{p,1}^{\frac{2}{p}+2})}
+\|\nabla\Pi\|_{L^1([\tau,t];\dot{B}_{p,1}^{\frac{2}{p}})}\nonumber\\
\leq&C\big(\|u(\tau)\|_{\dot{B}_{p,1}^{\frac{2}{p}}}+2^{m}\|a\|_{L_t^\infty(\dot{B}_{p,1}^{\frac{2}{p}})}Z(t)\big)e^{CZ(t)}.
\end{align*}
Integrating the above inequality for $\tau$ over $(0,t_0)$, we conclude the proof of \eqref{hr}.
\end{proof}

\subsection{Energy estimates in the $L^2$ framework}
Let $T^*$ be the maximal existence time of the unique local solution $(a,u,\nabla\Pi)$ obtained in Section \ref{Section4}.
Thanks to \eqref{soln regularity} and \eqref{hr}, we infer that for any $t\in(0,T^*)$, there exists a time $t_1\in(0,t)$ such that
$u(t_1)\in\dot{B}_{p,1}^{\frac2p-1}(\R^2)\cap\dot{B}_{p,1}^{\frac2p+2}(\R^2)$. As in \eqref{Modelbar}, denote by $u_F(t)\stackrel{\mathrm{def}}{=}e^{(t-t_1)\Delta}u(t_1)$.
Then thanks to Lemma \ref{Heatflow}, we have for $s\in[\frac2p-1,\frac2p+2]$ that
\begin{align}\label{guF}
\|u_F\|_{\widetilde{L}^\infty([t_1,\infty);\dot{B}_{p,1}^{s})}+\|u_F\|_{L^1([t_1,\infty);\dot{B}_{p,1}^{s+2})}
+\|u_F\|_{L^2([t_1,\infty);\dot{B}_{p,1}^{s+1})}\lesssim\|u(t_1)\|_{\dot{B}_{p,1}^{s}}.
\end{align}

Note that for $p\in(2,4)$, the free solution $u_F$ is not of finite energy. Fortunately, the convection term $u_F\cdot\nabla u_F$ is of finite energy.
Indeed, for $u,v\in\dot{B}_{p,1}^{\frac2p-1}(\R^2)\cap\dot{B}_{p,1}^{\frac2p+1}(\R^2)$ with $p\in(2,4)$, we have
\begin{align}\label{gconvection}
\|u\cdot\nabla v\|_{L^2}\leq&\|u\|_{L^4}\|\nabla v\|_{L^4}
\leq\|u\|_{\dot{B}_{p,1}^{\frac2p-1}}^\frac34\|u\|_{\dot{B}_{p,1}^{\frac2p+1}}^\frac14
\|v\|_{\dot{B}_{p,1}^{\frac2p-1}}^\frac14\|v\|_{\dot{B}_{p,1}^{\frac2p+1}}^\frac34\nonumber\\
\leq&\|u\|_{\dot{B}_{p,1}^{\frac2p-1}}\|v\|_{\dot{B}_{p,1}^{\frac2p+1}}
+\|u\|_{\dot{B}_{p,1}^{\frac2p+1}}\|v\|_{\dot{B}_{p,1}^{\frac2p-1}}.
\end{align}

Owing to \eqref{Modelbar}, we next present the $L^2$ energy estimates for $\bar{u}$ in the case when $p\in(2,4)$. Similar estimates for $p\in(1,2]$ will be mentioned after Lemma \ref{Lemma5.3}.

\begin{Lemma}\label{Lemma5.1}($L^2$ estimate of $\bar{u}$).
Under the assumptions of Theorem \ref{Theorem1.1}, there exists a  time independent constant $C$ such that
\begin{align}\label{gL2}
\|\bar{u}\|_{L^\infty([t_1,T^*);L^2)}+\|\nabla \bar{u}\|_{L^2([t_1,T^*);L^2)}\leq C.
\end{align}
\end{Lemma}
\begin{proof} Firstly thanks $1+a_0\geq\kappa$, we deduce from the transport equation of \eqref{Modelbar} that
\begin{align}\label{grho}
\frac{1}{1+\|a_0\|_{L^\infty}}\leq\rho(t,x)\leq\frac{1}{\kappa}.
\end{align}
Next, taking the $L^2$ inner product of the $\bar{u}$ equation of \eqref{Modelbar} with $\bar{u}$ leads to
\begin{align}\label{g1}
\frac12\frac{d}{dt}\|\sqrt{\rho}\bar{u}\|_{L^2}^2+\|\nabla \bar{u}\|_{L^2}^2=\int_{\R^2}\bar{u}\cdot Gdx
\lesssim\|\sqrt{\rho}\bar{u}\|_{L^2}\|G\|_{L^2}.
\end{align}
Yet thanks to $1-\rho=\rho a$ and \eqref{grho}, we get by applying H$\rm{\ddot{o}}$lder inequality and \eqref{gconvection} that
\begin{align}\label{gG}
\|G\|_{L^2}
\lesssim&\|a\|_{L^{\frac{2p}{p-2}}}\|\Delta u_F\|_{L^p}+\|u_F\cdot\nabla u_F\|_{L^2}+\|\sqrt{\rho}\bar{u}\|_{L^2}\|\nabla u_F\|_{L^\infty}\nonumber\\
\lesssim&\|u_F\|_{\dot{B}_{p,1}^{2}}+\|u_F\|_{\dot{B}_{p,1}^{\frac2p-1}}\|u_F\|_{\dot{B}_{p,1}^{\frac2p+1}}
+\|\sqrt{\rho}\bar{u}\|_{L^2}\|u_F\|_{\dot{B}_{p,1}^{\frac2p+1}}.
\end{align}
Thanks to \eqref{guF}, plugging \eqref{gG} into \eqref{g1} and applying Gronwall's inequality gives rise to \eqref{gL2}.
This completes the proof of the lemma.
\end{proof}

\begin{Lemma}\label{Lemma5.2}($H^1$ estimate of $\bar{u}$).
Under the assumptions of Theorem \ref{Theorem1.1}, there exists a time independent constant $C$ such that
\begin{align}\label{gH1}
\|\nabla \bar{u}\|_{L^\infty([t_1,T^*);L^2)}+\|\bar{u}_t,\Delta \bar{u},\nabla\Pi\|_{L^2([t_1,T^*);L^2)}\leq C.
\end{align}
\end{Lemma}
\begin{proof} Taking the $L^2$ inner product of the $\bar{u}$ equation of \eqref{Modelbar} with $\bar{u}_t$ gives
\begin{align}\label{g2}
\frac{d}{dt}\|\nabla \bar{u}\|_{L^2}^2+\|\sqrt{\rho}\bar{u}_t\|_{L^2}^2
\lesssim\|G\|_{L^2}^2+\|\rho(u_F+\bar{u})\cdot\nabla \bar{u}\|_{L^2}^2.
\end{align}
On the other hand,  we readily deduce from \eqref{Modelbar} that
\begin{align}\label{guPi}
\|\Delta \bar{u}\|_{L^2}^2+\|\nabla\Pi\|_{L^2}^2=&\|\Delta \bar{u}-\nabla\Pi\|_{L^2}^2
\lesssim\|\sqrt{\rho}\bar{u}_t\|_{L^2}^2+\|\rho(u_F+\bar{u})\cdot\nabla \bar{u}\|_{L^2}^2+\|G\|_{L^2}^2,
\end{align}
which along with \eqref{g2} implies
\begin{align}\label{g3}
\frac{d}{dt}\|\nabla \bar{u}\|_{L^2}^2+C^{-1}\|\sqrt{\rho}\bar{u}_t,\Delta \bar{u},\nabla\Pi\|_{L^2}^2
\lesssim\|G\|_{L^2}^2+\|u_F\|_{L^\infty}^2\|\nabla \bar{u}\|_{L^2}^2+\|\bar{u}\cdot\nabla \bar{u}\|_{L^2}^2.
\end{align}
While taking advantage of Gagliardo-Nirenberg inequality, we obtain
\begin{align}\label{g4}
\|\bar{u}\cdot\nabla \bar{u}\|_{L^2}^2\lesssim\|\bar{u}\|_{L^4}^2\|\nabla \bar{u}\|_{L^4}^2\lesssim
\|\bar{u}\|_{L^2}\|\nabla \bar{u}\|_{L^2}^2\|\Delta \bar{u}\|_{L^2}.
\end{align}
Plugging \eqref{gG} and \eqref{g4} into \eqref{g3} and using Young's inequality leads to
\begin{align*}
&\frac{d}{dt}\|\nabla \bar{u}\|_{L^2}^2+C^{-1}\|\sqrt{\rho}\bar{u}_t,\Delta \bar{u},\nabla\Pi\|_{L^2}^2\nonumber\\
\lesssim&\big(\|u_F\|_{\dot{B}_{p,1}^{\frac2p}}^2+\|\bar{u}\|_{L^2}^2\|\nabla \bar{u}\|_{L^2}^2\big)\|\nabla \bar{u}\|_{L^2}^2
+\big(\|u_F\|_{\dot{B}_{p,1}^{\frac2p-1}}^2+\|\bar{u}\|_{L^2}^2\big)\|u_F\|_{\dot{B}_{p,1}^{\frac2p+1}}^2+\|u_F\|_{\dot{B}_{p,1}^{2}}^2,
\end{align*}
from which, \eqref{guF} and \eqref{gL2}, we conclude the proof of \eqref{gH1}.
\end{proof}

\begin{Lemma}\label{Lemma5.3}($H^2$ estimate of $\bar{u}$).
Under the assumptions of Theorem \ref{Theorem1.1}, there exists a time independent constant $C$ such that for any $q\in[2,\infty)$
\begin{align}\label{gH2}
\|\bar{u}_t,\Delta \bar{u},\nabla\Pi\|_{L^\infty([t_1,T^*);L^2)}+\|\nabla \bar{u}_t\|_{L^2([t_1,T^*);L^2)}
+\|\Delta\bar{u},\nabla\Pi\|_{L^2([t_1,T^*);L^q)}\leq C.
\end{align}
\end{Lemma}
\begin{proof} Firstly, applying $\partial_t$ to the $\bar{u}$ equation of \eqref{Modelbar} and then
taking the $L^2$ inner product of the resulting equation with $\bar{u}_t$, we obtain
\begin{align}\label{g5}
&\frac12\frac{d}{dt}\|\sqrt{\rho}\bar{u}_t\|_{L^2}^2+\|\nabla \bar{u}_t\|_{L^2}^2\nonumber\\
=&\int_{\R^2}(1-\rho)\bar{u}_t\cdot\Delta^2 u_F dx\nonumber\\
&-\int_{\R^2}\rho_t \bar{u}_t\cdot\big(\bar{u}_t+\bar{u}\cdot\nabla \bar{u}+u_F\cdot\nabla \bar{u}
+\bar{u}\cdot\nabla u_F+\Delta u_F+u_F\cdot\nabla u_F\big)dx\nonumber\\
&-\int_{\R^2}\rho \bar{u}_t\cdot\big(\bar{u}_t\cdot\nabla\bar{u}+\Delta u_F\cdot\nabla\bar{u}+\bar{u}_t\cdot\nabla u_F
+\bar{u}\cdot\nabla\Delta u_F+\partial_t(u_F\cdot\nabla u_F)\big)dx\nonumber\\
\stackrel{\rm{def}}{=}&I_1+I_2+I_3.
\end{align}
Using $\rho_t=-\diverg(\rho(u_F+\bar{u}))$ and integration by parts gives
\begin{align*}
I_2=&-2\int_{\R^2}\rho \bar{u}_t\cdot\big((u_F+\bar{u})\cdot\nabla \bar{u}_t\big)dx-\int_{\R^2}\rho(u_F+\bar{u})^j\bar{u}_t\cdot\big
\{(u_F+\bar{u})\cdot\nabla\partial_j\bar{u}\nonumber\\
&+\partial_j\bar{u}\cdot\nabla\bar{u}+\partial_ju_F\cdot\nabla\bar{u}+\partial_j \bar{u}\cdot\nabla u_F+\bar{u}\cdot\nabla \partial_j u_F
+\partial_j\Delta u_F+\partial_j(u_F\cdot\nabla u_F)\big\}dx\nonumber\\
&-\int_{\R^2}\rho\big((u_F+\bar{u})\cdot\nabla\bar{u}_t\big)\cdot\big((u_F+\bar{u})\cdot\nabla\bar{u}
+\bar{u}\cdot\nabla u_F+\Delta u_F+u_F\cdot\nabla u_F\big)dx\nonumber\\
\stackrel{\rm{def}}{=}&I_2^1+I_2^2+I_2^3.
\end{align*}
Then it is easy to observe that
\begin{align}\label{gI1}
I_1\lesssim\|a\|_{L^{\frac{2p}{p-2}}}\|\Delta^2u_F\|_{L^p}\|\sqrt{\rho}\bar{u}_t\|_{L^2}
\lesssim\|u_F\|_{\dot{B}_{p,1}^{4}}\|\sqrt{\rho}\bar{u}_t\|_{L^2}.
\end{align}
Applying H${\rm\ddot{o}}$lder's inequality and Gagliardo-Nirenberg inequality, we infer for any $\eta>0$ that
\begin{align}\label{gI21}
I_2^1\lesssim&\|\sqrt{\rho}\bar{u}_t\|_{L^2}\|u_F\|_{L^\infty}\|\nabla\bar{u}_t\|_{L^2}
+\|\bar{u}_t\|_{L^4}\|\bar{u}\|_{L^4}\|\nabla \bar{u}_t\|_{L^2}\nonumber\\
\lesssim&\|\sqrt{\rho}\bar{u}_t\|_{L^2}\|u_F\|_{\dot{B}_{p,1}^{\frac2p}}\|\nabla \bar{u}_t\|_{L^2}
+\|\sqrt{\rho}\bar{u}_t\|_{L^2}^{\frac12}\|\bar{u}\|_{L^2}^{\frac12}\|\nabla \bar{u}\|_{L^2}^{\frac12}\|\nabla \bar{u}_t\|_{L^2}^{\frac32}\nonumber\\
\lesssim&\eta\|\nabla \bar{u}_t\|_{L^2}^2+\frac{1}{\eta^3}\big(\|u_F\|_{\dot{B}_{p,1}^{\frac2p}}^2
+\|\bar{u}\|_{L^2}^2\|\nabla\bar{u}\|_{L^2}^2\big)\|\sqrt{\rho}\bar{u}_t\|_{L^2}^2.
\end{align}
Yet thanks to \eqref{guF}, \eqref{gL2} and \eqref{gH1}, we infer for any $p_1\in[p,\infty]$ and $p_2\in[2,\infty)$ that
$$
\|u_F\|_{L^{p_1}}\leq C\|u_F\|_{\dot{B}_{p,1}^{\frac2p-\frac{2}{p_1}}}\leq C\ \ \mathrm{and}\ \
\|\bar{u}\|_{L^{p_2}}\leq C\|\bar{u}\|_{H^1}\leq C,
$$
which together with a similar argument as \eqref{gI21} results in
\begin{align}
I_2^2\lesssim&
\|\sqrt{\rho}\bar{u}_t\|_{L^2}\big\{\|u_F+\bar{u}\|_{L^\infty}^2\|\nabla^2 \bar{u}\|_{L^2}+\|u_F+\bar{u}\|_{L^6}\|\nabla \bar{u}\|_{L^6}^2\nonumber\\
&+\|u_F+\bar{u}\|_{L^4}\|\nabla u_F\|_{L^\infty}\|\nabla \bar{u}\|_{L^4}+\|u_F+\bar{u}\|_{L^4}\|\bar{u}\|_{L^4}\|\nabla^2 u_F\|_{L^\infty}\nonumber\\
&+\|u_F+\bar{u}\|_{L^\frac{2p}{p-2}}\big(\|\nabla^3u_F\|_{L^p}+\|\nabla u_F\|_{L^{2p}}^2
+\|u_F\|_{L^p}\|\nabla^2u_F\|_{L^\infty}\big)\big\}\nonumber\\
\lesssim&\|\sqrt{\rho}\bar{u}_t\|_{L^2}\big\{\|u_F\|_{L^\infty}\|\Delta\bar{u}\|_{L^2}+\|\bar{u}\|_{L^2}\|\Delta\bar{u}\|_{L^2}^2
+\|\nabla\bar{u}\|_{L^2}^\frac23\|\Delta\bar{u}\|_{L^2}^\frac43\nonumber\\
&+\|\nabla u_F\|_{L^\infty}\|\nabla\bar{u}\|_{L^2}^\frac12\|\Delta\bar{u}\|_{L^2}^\frac12
+\|\nabla^2 u_F\|_{L^\infty}+\|\nabla^3u_F\|_{L^p}+\|\nabla u_F\|_{L^{2p}}^2\big\}\nonumber\\
\lesssim&\|\sqrt{\rho}\bar{u}_t\|_{L^2}\big\{\|\Delta\bar{u}\|_{L^2}^2+\|\nabla\bar{u}\|_{L^2}^2+\|u_F\|_{\dot{B}_{p,1}^{\frac2p}}^2
+\|u_F\|_{\dot{B}_{p,1}^{\frac2p+1}}^2+\|u_F\|_{\dot{B}_{p,1}^{\frac2p+2}}\nonumber\\
&+\|u_F\|_{\dot{B}_{p,1}^{3}}+\|u_F\|_{\dot{B}_{p,1}^{\frac1p+1}}^2\big\}.
\end{align}
Along the same line, one has
\begin{align}
I_2^3\lesssim&
\|\nabla\bar{u}_t\|_{L^2}\big\{\|u_F+\bar{u}\|_{L^6}^2\|\nabla \bar{u}\|_{L^6}+\|u_F+\bar{u}\|_{L^4}\|\bar{u}\|_{L^4}\|\nabla u_F\|_{L^\infty}\nonumber\\
&+\|u_F+\bar{u}\|_{L^\frac{2p}{p-2}}\big(\|\Delta u_F\|_{L^p}+\|u_F\|_{L^p}\|\nabla u_F\|_{L^\infty}\big)\big\}\nonumber\\
\lesssim&\eta\|\nabla \bar{u}_t\|_{L^2}^2+\frac1\eta\big(\|\nabla \bar{u}\|_{L^2}^2+\|\Delta\bar{u}\|_{L^2}^2
+\|u_F\|_{\dot{B}_{p,1}^{\frac2p+1}}^2+\|u_F\|_{\dot{B}_{p,1}^{2}}^2\big).
\end{align}
Finally thanks to \eqref{gconvection}, \eqref{gL2} and \eqref{gH1}, using the same argument leads to
\begin{align}\label{gI3}
I_3\lesssim&\|\bar{u}_t\|_{L^4}^2\|\nabla \bar{u}\|_{L^2}+\|\sqrt{\rho}\bar{u}_t\|_{L^2}\|\Delta u_F\|_{L^\infty}\|\nabla\bar{u}\|_{L^2}
+\|\sqrt{\rho}\bar{u}_t\|_{L^2}^2\|\nabla u_F\|_{L^\infty}\nonumber\\
&+\|\sqrt{\rho}\bar{u}_t\|_{L^2}\|\bar{u}\|_{L^2}\|\nabla \Delta u_F\|_{L^\infty}
+\|\sqrt{\rho}\bar{u}_t\|_{L^2}\|\partial_t(u_F\cdot\nabla u_F)\|_{L^2}
\nonumber\\
\lesssim&\|\bar{u}_t\|_{L^2}\|\nabla\bar{u}_t\|_{L^2}\|\nabla \bar{u}\|_{L^2}
+\|\sqrt{\rho}\bar{u}_t\|_{L^2}^2\|\nabla u_F\|_{L^\infty}\nonumber\\
&+\|\sqrt{\rho}\bar{u}_t\|_{L^2}\big(\|\Delta u_F\|_{L^\infty}+\|\nabla \Delta u_F\|_{L^\infty}
+\|u_F\|_{\dot{B}_{p,1}^{\frac2p+3}}\|u_F\|_{\dot{B}_{p,1}^{\frac2p-1}}\big)
\nonumber\\
\lesssim&\eta\|\nabla\bar{u}_t\|_{L^2}^2+
\|\sqrt{\rho}\bar{u}_t\|_{L^2}^2\big(\frac1\eta\|\nabla\bar{u}\|_{L^2}^2+\|u_F\|_{\dot{B}_{p,1}^{\frac2p+1}}\big)\nonumber\\
&+\|\sqrt{\rho}\bar{u}_t\|_{L^2}\big(\|u_F\|_{\dot{B}_{p,1}^{\frac2p+2}}+\|u_F\|_{\dot{B}_{p,1}^{\frac2p+3}}\big).
\end{align}
Thus, plugging \eqref{gI1}--\eqref{gI3} into \eqref{g5} and taking $\eta$ small enough yields
\begin{align}\label{gIsum}
&\frac{d}{dt}\|\sqrt{\rho}\bar{u}_t\|_{L^2}^2+\|\nabla \bar{u}_t\|_{L^2}^2\nonumber\\
\lesssim&\|\sqrt{\rho}\bar{u}_t\|_{L^2}\big(\|\Delta\bar{u}\|_{L^2}^2+\|\nabla\bar{u}\|_{L^2}^2
+\|u_F\|_{\dot{B}_{p,1}^{\frac2p+1}}+\|u_F\|_{\dot{B}_{p,1}^{4}}\big)\nonumber\\
&+\|\sqrt{\rho}\bar{u}_t\|_{L^2}^2\big(\|\nabla\bar{u}\|_{L^2}^2+\|u_F\|_{\dot{B}_{p,1}^{\frac2p+1}}\big)
+\|\nabla \bar{u}\|_{L^2}^2+\|\Delta\bar{u}\|_{L^2}^2+\|u_F\|_{\dot{B}_{p,1}^{\frac2p+1}}^2+\|u_F\|_{\dot{B}_{p,1}^{2}}^2\nonumber\\
\lesssim&\big(1+\|\sqrt{\rho}\bar{u}_t\|_{L^2}^2\big)\big(\|\Delta\bar{u}\|_{L^2}^2+\|\nabla\bar{u}\|_{L^2}^2
+\|u_F\|_{\dot{B}_{p,1}^{\frac2p+1}}+\|u_F\|_{\dot{B}_{p,1}^{4}}\big).
\end{align}
where we used the fact that $\|u_F\|_{\dot{B}_{p,1}^{s}}\leq\|u_F\|_{\dot{B}_{p,1}^{s_1}}+\|u_F\|_{\dot{B}_{p,1}^{s_2}}$ for $s\in[s_1,s_2]$.

Whereas taking the $L^2$ inner product of the $\bar{u}$ equation of \eqref{Modelbar} with $\bar{u}_t$ at $t=t_1$
and using \eqref{guF} and \eqref{gconvection} results in
\begin{align*}
&\|(\sqrt{\rho}\bar{u}_t)(t_1)\|_{L^2}\leq C\|a(t_1)\|_{L^\frac{2p}{p-2}}\|\Delta u_F(t_1)\|_{L^p}+C\|u_F\cdot\nabla u_F(t_1)\|_{L^2}\\
\leq&C\|u(t_1)\|_{\dot{B}_{p,1}^{2}}+C\|u(t_1)\|_{\dot{B}_{p,1}^{\frac2p-1}}\|u(t_1)\|_{\dot{B}_{p,1}^{\frac2p+1}}\leq C.
\end{align*}
As a consequence, applying Gronwall's inequality to \eqref{gIsum} and taking advantage of \eqref{guF}, \eqref{gL2} and \eqref{gH1} gives rise to
\begin{align}\label{gH21}
\|\bar{u}_t\|_{L^\infty([t_1,T^*);L^2)}+\|\nabla \bar{u}_t\|_{L^2([t_1,T^*);L^2)}\leq C.
\end{align}

Owing to \eqref{Modelbar}, we can derive the second space derivative estimate of $\bar{u}$.
Indeed, we deduce from \eqref{gG}, \eqref{guPi} and \eqref{g4} that
\begin{align*}
\|\Delta \bar{u}\|_{L^2}^2+\|\nabla\Pi\|_{L^2}^2
\lesssim&\|\bar{u}_t\|_{L^2}^2+\big(1+\|\nabla\bar{u}\|_{L^2}^2\big)\|u_F\|_{\dot{B}_{p,1}^{\frac2p-1}}\|u_F\|_{\dot{B}_{p,1}^{\frac2p+1}}\\
&+\|\bar{u}\|_{L^2}^2\|\nabla \bar{u}\|_{L^2}^4+\|u_F\|_{\dot{B}_{p,1}^{2}}+\|\bar{u}\|_{L^2}\|u_F\|_{\dot{B}_{p,1}^{\frac2p+1}},
\end{align*}
which along with \eqref{guF}, \eqref{gL2}, \eqref{gH1} and \eqref{gH21} ensures that
\begin{align}\label{gH22}
\|\Delta \bar{u}\|_{L^\infty([t_1,T^*);L^2)}+\|\nabla\Pi\|_{L^\infty([t_1,T^*);L^2)}\leq C.
\end{align}
While thanks to $\diverg\bar{u}=0$, we deduce from the $\bar{u}$ equation of \eqref{Modelbar} for $q\in[p,\infty)$ that
\begin{align*}
\|\Delta\bar{u}\|_{L^q}+\|\nabla\Pi\|_{L^q}
\lesssim&\|\bar{u}_t\|_{L^q}+\|u_F+\bar{u}\|_{L^{2q}}\|\nabla(u_F+\bar{u})\|_{L^{2q}}+\|\Delta u_F\|_{L^q}\nonumber\\
\lesssim&\|\bar{u}_t\|_{L^2}^{\frac2q}\|\nabla\bar{u}_t\|_{L^2}^{1-\frac2q}+\|\nabla\bar{u}\|_{L^2}^{\frac1q}\|\Delta\bar{u}\|_{L^2}^{1-\frac1q}\nonumber\\
&+\|u_F\|_{\dot{B}_{p,1}^{1+\frac2p-\frac1q}}+\|u_F\|_{\dot{B}_{p,1}^{2(1+\frac1p-\frac1q)}},
\end{align*}
which along with \eqref{guF}, \eqref{gL2}, \eqref{gH1} and \eqref{gH21} implies
\begin{align}\label{gH23}
\|\Delta\bar{u}\|_{L^2([t_1,T^*);L^q)}+\|\nabla\Pi\|_{L^2([t_1,T^*);L^q)}\leq C.
\end{align}
Combining \eqref{gH21}--\eqref{gH23} and \eqref{gH1} gives the results.
\end{proof}

\begin{Remark}
For $p\in(1,2]$, we have $\dot{B}_{p,1}^{\frac2p-1}(\R^2)\hookrightarrow L^2(\R^2)$. Instead of using the decomposition $u=u_F+\bar{u}$, we could directly present the $L^2$ energy estimates for $u$. With some slight modifications of the proof of Lemmas \ref{Lemma5.1}--\ref{Lemma5.3}, we could deduce from \eqref{Modelrho} for $q\in[2,\infty)$ that
\begin{align}\label{gH2u}
\|u_t,\Delta u,\nabla\Pi\|_{L^\infty([t_1,T^*);L^2)}+\|\nabla u_t\|_{L^2([t_1,T^*);L^2)}
+\|\Delta u,\nabla\Pi\|_{L^2([t_1,T^*);L^q)}\leq C.
\end{align}
\end{Remark}

\subsection{Proof of the global well-posedness part of Theorem \ref{Theorem1.1}}
Firstly, thanks to Lemma \ref{Bernstein Lemma}, one has for any $q\in(2,\infty)$
\begin{align*}
\|\bar{u}\|_{\dot{B}_{q,1}^{\frac{2}{q}+1}}\lesssim\|\nabla\bar{u}\|_{L^q}^{1-\frac2q}\|\Delta \bar{u}\|_{L^q}^{\frac2q}
\lesssim\|\nabla\bar{u}\|_{H^1}^{1-\frac2q}\|\Delta \bar{u}\|_{L^q}^{\frac2q},
\end{align*}
which together with \eqref{guF}, \eqref{gL2},\eqref{gH1},\eqref{gH2} and \eqref{gH2u} results in
\begin{align*}
\|u\|_{L^1([t_1,t];\dot{B}_{q,1}^{\frac{2}{q}+1})}\leq Ct^\frac12,\quad t< T^*.
\end{align*}
Then for $q\in(2,\infty)$ with $\frac1q\geq\frac1p-\frac12$,
we get by applying Proposition \ref{Proposition2.1} to the transport equation of \eqref{Modela} that
\begin{align}\label{ga}
\|a\|_{\widetilde{L}^\infty([t_1,t];B_{p,1}^{\frac{2}{p}})}
\leq\|a(t_1)\|_{B_{p,1}^{\frac{2}{p}}}\exp\left(C\|u\|_{L^1([t_1,t];\dot{B}_{q,1}^{\frac{2}{q}+1})}\right)\leq C\exp\big(Ct^\frac12\big).
\end{align}
On the other hand, we rewrite the equation for $u$ in \eqref{Modela} as
$$
\partial_t u-\Delta u+\nabla\Pi=\frac{a}{1+a}\partial_t u-\frac{1}{1+a}u\cdot\nabla u.
$$
Thanks to $\diverg u=0$, it is easy to observe that for $t\in[t_1,T^*)$
\begin{align}\label{gstokes}
&\|u\|_{\widetilde{L}^\infty([t_1,t];\dot{B}_{p,1}^{\frac{2}{p}-1})}+\|u\|_{L^1([t_1,t];\dot{B}_{p,1}^{\frac{2}{p}+1})}
+\|\nabla\Pi\|_{L^1([t_1,t];\dot{B}_{p,1}^{\frac{2}{p}-1})}\nonumber\\
\lesssim&\|u(t_1)\|_{\dot{B}_{p,1}^{\frac{2}{p}-1}}+\big\|\frac{a}{1+a}\partial_tu\big\|_{L^1([t_1,t];\dot{B}_{p,1}^{\frac{2}{p}-1})}
+\big\|\frac{1}{1+a}u\cdot\nabla u\big\|_{L^1([t_1,t];\dot{B}_{p,1}^{\frac{2}{p}-1})}.
\end{align}
Yet applying Lemma \ref{Product Laws} leads to
\begin{align}\label{g6}
\big\|\frac{1}{1+a}u\cdot\nabla u\big\|_{L^1([t_1,t];\dot{B}_{p,1}^{\frac{2}{p}-1})}
\leq& C\big(1+\|a\|_{L^\infty([t_1,t];\dot{B}_{p,1}^{\frac{2}{p}})}\big)
\int_{t_1}^{t}\|u\|_{\dot{B}_{p,1}^{\frac{2}{p}-1}}\|u\|_{\dot{B}_{q,1}^{\frac{2}{q}+1}}dt'\nonumber\\
\leq& C\exp\big(Ct^\frac12\big)\int_{t_1}^{t}\|u\|_{\dot{B}_{p,1}^{\frac{2}{p}-1}}\|u\|_{\dot{B}_{q,1}^{\frac{2}{q}+1}}dt'.
\end{align}
where $q\in(2,\infty)$ satisfies $\frac1q>\frac12-\frac1p$ and $\frac1q\geq\frac1p-\frac12$. To control the second term in the right hand side of \eqref{gstokes}, we use again Lemma \ref{Product Laws} to get for $p\in(1,2]$ and $\alpha\in(\frac2p-1,1)$ that
\begin{align}
\big\|\frac{a}{1+a}\partial_t u\big\|_{L^1([t_1,t];\dot{B}_{p,1}^{\frac{2}{p}-1})}
\leq&C\|a\|_{L^\infty([t_1,t];\dot{B}_{p,1}^{\frac{2}{p}-\alpha})}\|\partial_tu\|_{L^1([t_1,t];\dot{B}_{2,1}^{\alpha})}\nonumber\\
\leq&Ct^\frac12\|a\|_{L^\infty([t_1,t];B_{p,1}^{\frac{2}{p}})}\|\partial_tu\|_{L^2([t_1,t];H^1)}\leq C\exp\big(Ct^\frac12\big).
\end{align}
While for $p\in(2,4)$ and $\alpha\in(0,\frac2p)$, we  alternatively get
\begin{align}\label{g7}
&\big\|\frac{a}{1+a}\partial_t u\big\|_{L^1([t_1,t];\dot{B}_{p,1}^{\frac{2}{p}-1})}\nonumber\\
\leq& C\|a\|_{L^\infty([t_1,t];\dot{B}_{p,1}^{\frac{2}{p}})}\|u_F\|_{L^1([t_1,t];\dot{B}_{p,1}^{\frac{2}{p}+1})}
+C\|a\|_{L^\infty([t_1,t];\dot{B}_{p,1}^{\frac{2}{p}-\alpha})}\|\partial_t\bar{u}\|_{L^1([t_1,t];\dot{B}_{2,1}^{\alpha})}\nonumber\\
\leq& C\exp\big(Ct^\frac12\big).
\end{align}
Plugging \eqref{g6}--\eqref{g7} into \eqref{gstokes} and then applying Gronwall's inequality leads to
\begin{align}\label{gfinal}
\|u\|_{\widetilde{L}^\infty([t_1,t];\dot{B}_{p,1}^{\frac{2}{p}-1})}+\|u\|_{L^1([t_1,t];\dot{B}_{p,1}^{\frac{2}{p}+1})}
+\|\nabla\Pi\|_{L^1([t_1,t];\dot{B}_{p,1}^{\frac{2}{p}-1})}\leq C\exp\left\{C\exp\big(Ct^\frac12\big)\right\}.
\end{align}
From \eqref{ga} and \eqref{gfinal}, we infer by a standard argument that $T^*=\infty$.

The proof of Theorem \ref{Theorem1.1} is completed.


\bigskip

\noindent{\large\bf  References}


\begin{thebibliography}{99}
\bibitem{Abidi RMI 2007}
   \newblock H. Abidi,
    \newblock
    \emph{\'{E}quation de Navier-Stokes avec densit\'{e} et viscosit\'{e} variables dans I'espace critique},
     \newblock Rev. Mat. Iberoam. {\bf 23} (2007) 537-586.

\bibitem{AGZ CPAM 2011}
   \newblock H. Abidi, G. Gui, P. Zhang,
    \newblock
    \emph{On the decay and stability of global solutions to the 3D inhomogeneous Navier-Stokes equations},
     \newblock Commun. Pure. Appl. Math. {\bf 64} (2011) 832-881.

\bibitem{AGZ ARMA 2012}
   \newblock H. Abidi, G. Gui, P. Zhang,
    \newblock
    \emph{On the wellposedness of three-dimensional inhomogeneous Navier-Stokes equations in the critical spaces},
     \newblock Arch. Rational. Mech. Anal. {\bf 204} (2012) 189-230.

\bibitem{AGZ JMPA 2013}
   \newblock H. Abidi, G. Gui, P. Zhang,
    \newblock
    \emph{Well-posedness of 3-D inhomogeneous Navier-Stokes equations with highly oscillatory initial velocity field},
     \newblock J. Math. Pures. Appl. {\bf 100} (2013) 166-203.

\bibitem{AP AIF 2007}
   \newblock H. Abidi, M. Paicu,
    \newblock
    \emph{Existence globale pour un fluide inhomog\'{e}ne},
     \newblock Ann. Inst. Fourier {\bf 57} (2007) 883-917.

\bibitem{AKM}
   \newblock S.N. Antontsev, A.V. Kazhikhov, V.N. Monakhov,
    \newblock
    \emph{Boundary value problem in mechanics of nonhomogeneous fluids},
     \newblock Stud. Math. Appl., vol. {\bf 22}, North-Holland Publishing Co., Amsterdam, 1990, translated from Russian.

\bibitem{BCD}
   \newblock H. Bahouri, J.Y. Chemin, R. Danchin,
    \newblock
    \emph{Fourier analysis and nonlinear partial differential equations},
     \newblock  Grundlehren Math. Wiss., vol. {\bf 343}, Springer (2011).

\bibitem{Bony}
   \newblock J.M. Bony,
    \newblock
    \emph{Calcul symbolique et propagation des singularit\'{e}s pour les \'{e}quations aux d\'{e}riv\'{e}es partielles non lin\'{e}aires},
     \newblock Ann. Scient. Ec. Norm. Super. {\bf 14} (1981) 209-246.

\bibitem{Danchin CPDE 2001}
   \newblock R. Danchin,
    \newblock
    \emph{Local theory in critical spaces for compressible viscous and heat-conductive gases},
     \newblock Comm. Partial Differential Equations, {\bf 26} (2001) 1183-1233.

\bibitem{Danchin PRSES 2003}
   \newblock R. Danchin,
    \newblock
    \emph{Density-dependent incompressible viscous fluids in critical spaces},
     \newblock Proc. R. Soc. Edin. Sect. A {\bf 133} (2003) 1311-1334.

\bibitem{Danchin ADE 2004}
   \newblock R. Danchin,
    \newblock
    \emph{Local and global well-posedness results for flows of inhomogeneous vicous fluids},
     \newblock Adv. Differential Equations {\bf 9} (2004) 353-386.

\bibitem{Danchin AFSTM 2006}
   \newblock R. Danchin,
    \newblock
    \emph{The inviscid limit for density-dependent incompressible fluids},
     \newblock Ann. Fac. Sci. Toulouse Math. {\bf 15} (2006) 637-688.

\bibitem{Danchin CPDE 2007}
   \newblock R. Danchin,
    \newblock
    \emph{Well-posedness in critical spaces for barotropic viscous fluids with truly not constant density},
     \newblock Comm. Partial Differential Equations, {\bf 32} (2007) 1373-1397.

\bibitem{Danchin JDE 2010}
   \newblock R. Danchin,
    \newblock
    \emph{On the well-posedness of the incompressible density-dependent Euler euations in the $L^p$ framework},
     \newblock J. Differential Equations, {\bf 248} (2010) 2130-2170.

\bibitem{Danchin AIF 2014}
   \newblock R. Danchin,
    \newblock
    \emph{A Lagrangian approach for the compressible Navier-Stokes equations},
     \newblock  Ann. Inst. Fourier, {\bf 64} (2014) 753-791.

\bibitem{DM CPAM 2012}
   \newblock R. Danchin, P.B. Mucha,
    \newblock
    \emph{A Lagrangian approach for the incompressible Navier-Stokes equations with variable density},
     \newblock Comm. Pure Appl. Math. {\bf 65} (2012) 1458-1480.

\bibitem{DM ARMA 2013}
   \newblock R. Danchin, P.B. Mucha,
    \newblock
    \emph{Incompressible flows with piecewise constant density},
     \newblock Arch. Rational. Mech. Anal. {\bf 207} (2013) 991-1023.

\bibitem{Desjardins DIE1}
   \newblock B. Desjardins,
    \newblock
    \emph{Global existence results for the incompressible density-dependent Navier-Stokes equations in the whole space},
     \newblock Differential and Integral Equations, {\bf 10} (1997) 587-598.

\bibitem{Desjardins DIE2}
   \newblock B. Desjardins,
    \newblock
    \emph{Linear transport equations with initial values in Sobolev spaces and application to the Navier-Stokes equations},
     \newblock Differential and Integral Equations, {\bf 10} (1997) 557-586.

\bibitem{Desjardins ARMA}
   \newblock B. Desjardins,
    \newblock
    \emph{Regularity results for two-dimensional flows of multiphase viscous fluids},
     \newblock Arch. Rational. Mech. Anal. {\bf 137} (1997) 135-158.

\bibitem{HPZ JMPA 2013}
   \newblock J. Huang, M. Paicu, P. Zhang,
    \newblock
    \emph{Global solutions to 2-D inhomogeneous Navier-Stokes system with general velocity},
     \newblock J. Math. Pures Appl. {\bf 100} (2013) 806-831.

\bibitem{HPZ ARMA 2013}
   \newblock J. Huang, M. Paicu, P. Zhang,
    \newblock
    \emph{Global well-posedness of incompressible inhomogeneous fluid systems with bounded density or non-Lipschitz velocity},
     \newblock Arch. Rational. Mech. Anal. {\bf 209} (2013) 631-682.

\bibitem{Ladyzhenskaya Solonnikov}
   \newblock O.A. Ladyzhenskaya, V.A. Solonnikov,
    \newblock
    \emph{The unique solvability of an initial-boundary value problem for viscous incompressible inhomogeneous fluids},
     \newblock Journal of Soviet Mathematics, {\bf 9} (1978) 697-749.

\bibitem{Lions}
   \newblock P.L. Lions,
    \newblock
    Mathematical Topics in Fluid Mechanics. vol. 1. Incompressible Models,
     \newblock Oxford Lecture Ser. Math. Appl., vol. 3, Oxford University Press, The Clarendon Press, New York, 1996.

\bibitem{PZZ CPDE 2013}
   \newblock M. Paicu, P. Zhang, Z. Zhang,
    \newblock
    \emph{Global unique solvability of inhomogeneous Navier-Stokes equations with bounded density},
     \newblock Commun. Partial Differential Equations {\bf 38} (2013) 1208-1234.

\bibitem{XLC}
   \newblock H. Xu, Y. Li, F. Chen,
    \newblock
    \emph{Global solution to the incompressible inhomogeneous Navier-Stokes equations with some large initial data},
     \newblock submitted.
\end{thebibliography}
\end{document}